\documentclass[preprint,12pt]{elsarticle}

\usepackage[normalem]{ulem}

\usepackage{xcolor}
\usepackage{amsmath,amssymb,amsfonts,amsthm,enumerate,multirow}
\usepackage{tikz}

\newtheorem{lemma}{Lemma}
\newtheorem{proposition}{Proposition}
\newtheorem{problem}{Problem}

\journal{Inverse Problems}

\begin{document}

\begin{frontmatter}

\title{Direct inversion of the Longitudinal Ray Transform for 2D residual elastic strain fields}

\author[Ncle]{CM Wensrich}
\cortext[Ncle]{Corresponding author:}
\ead{christopher.wensrich@newcastle.edu.au}
\author[Mchr]{S Holman}
\author[Chile]{M Courdurier}
\author[Mchr]{WRB Lionheart}
\author[Sobo]{AP Polyakova}
\author[Sobo]{IE Svetov}

\affiliation[Ncle]{organization={School of Engineering, University of Newcastle, Australia},
            addressline={University Drive}, 
            city={Callaghan},
            postcode={2308}, 
            state={NSW},
            country={Australia}}

\affiliation[Mchr]{organization={School of Mathematics, University of Manchester},
            addressline={Alan Turing Building, Oxford Rd}, 
            city={Manchester},
            postcode={M13 9PL}, 
            country={UK}}

\affiliation[Chile]{organization={Department of Mathematics, Pontificia Universidad Católica de Chile},
            addressline={Avda. Vicuña Mackenna 4860}, 
            city={Macul, Santiago},
            country={Chile}}

\affiliation[Sobo]{organization={Sobolev Institute of Mathematics, Novosibirsk State University},
            addressline={ 630090}, 
            city={Novosibirsk},
            country={Russia}}

\begin{abstract}
We examine the problem of Bragg-edge elastic strain tomography from energy resolved neutron transmission imaging.  A new approach is developed for two-dimensional plane-stress and plane-strain systems whereby elastic strain can be reconstructed from its Longitudinal Ray Transform (LRT) as two parts of a Helmholtz decomposition based on the concept of an Airy stress potential.  The solenoidal component of this decomposition is reconstructed using an inversion formula based on a tensor filtered back projection algorithm whereas the potential part can be recovered using either Hooke's law or a finite element model of the elastic system.  The technique is demonstrated for two-dimensional plane-stress systems in both simulation, and on real experimental data.  We also demonstrate that application of the standard scalar filtered back projection algorithm to the LRT in these systems recovers the trace of the solenoidal component of strain and we provide physical meaning for this quantity in the case of 2D plane-stress and plane-strain systems.

\end{abstract}



\begin{keyword}
Strain tomography \sep Longitudinal Ray Transform \sep Bragg edge \sep Neutron transmission



\end{keyword}

\end{frontmatter}


    %

\section{Introduction and context}

Elastic strain imaging via energy-resolved neutron transmission measurement (also known as `Bragg-edge imaging') forms a natural tensor-tomography problem aimed at reconstructing the full triaxial elastic strain field within a physical sample from a set of lower-dimensional scalar images.  

The full solution to this tomography problem will have a key impact in a number of areas in science and engineering focused on the study of residual stress in materials.  An important topical example includes the development of additive manufacturing techniques for metallic components where residual stresses generated by the thermo-mechanics of deposition are a significant and ever present concern.  Tomographic techniques for strain have the potential to provide a unique insight in this area.

The Bragg-edge strain tomography problem has been studied for more than a decade, with various experimental demonstrations on special cases (e.g. axisymmetric systems and \textit{in situ} applied loads) (e.g. \cite{hendriks2017bragg, abbey2012neutron, kirkwood2015neutron}), and, more recently, solutions for general systems using Bayesian and least-squares techniques constrained by equilibrium (e.g. \cite{gregg2018tomographic, hendriks2019tomographic}).  In this paper we examine this problem from the perspective of developing a direct inversion algorithm.

With reference to Figure \ref{LRTGeom}, strain images of this type refer to projections of the average of elastic strain, $\epsilon$, along straight-line ray paths through a sample $\Omega$ of the form
\begin{equation}
\frac{1}{L} \int_{-\infty}^\infty \epsilon_{ij}(x_0+s\xi)\xi_{i}\xi_{j}ds,
\label{BEStrain}
\end{equation}
where $L$ is the path-length associated with a ray passing through the point $x_0 \in \Omega$, travelling in the direction $\xi$, and, as in the rest of the paper, we use the summation convention for repeated indices. For convenience, strain outside of the boundary of the sample is assigned a value of zero. From many measurements of this form, we wish to reconstruct the original strain field. 

\begin{figure}[h]
\begin{center}
    \label{LRTGeom}
    \includegraphics[width=0.3\linewidth]{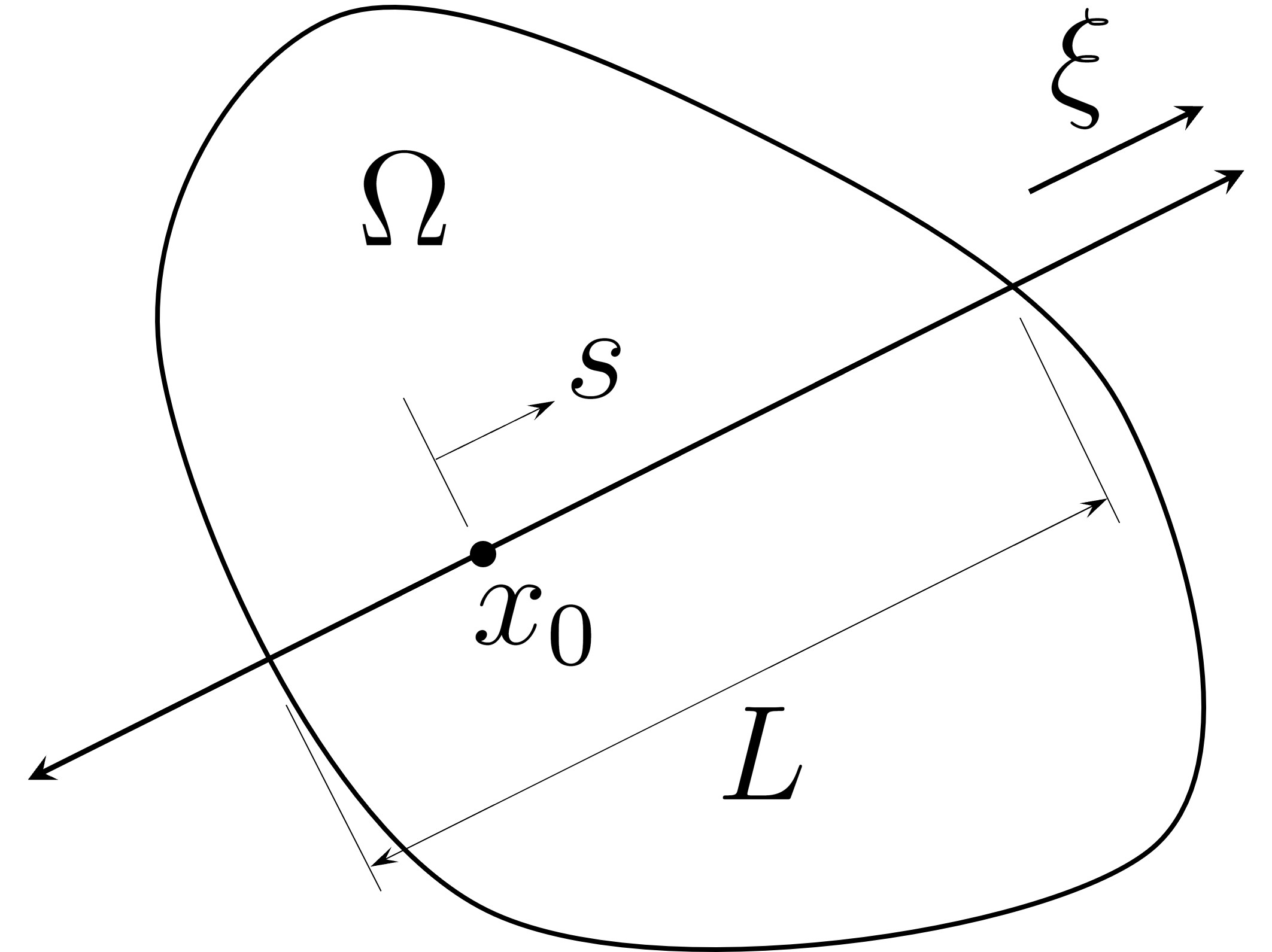}
    \caption{Geometry of the Longitudinal Ray Transform and Bragg-edge strain measurements.}
\end{center}
\end{figure} 

Bragg-edge strain measurements are naturally related to the Longitudinal Ray Transform (LRT), $I$, which can be written for suitable $f \in L^2(\mathcal{S}^m;\mathbb{R}^n)$ as
\begin{equation}
If(x_0,\xi)=\int_{-\infty}^\infty f_{{i_1}{i_2}...{i_m}}(x_0+s\xi)\xi_{i_1}\xi_{i_2}...\xi_{i_m} ds,
\end{equation}
with the extension to all of $L^2(\mathcal{S}^m;\mathbb{R}^n)$ achieved in the usual way (see below for definitions and notation). 

Unfortunately, the LRT has a large null space that creates a well-known issue with direct tomographic reconstruction of strain from Bragg-edge imaging and LRT measurements in general \cite{lionheart2015diffraction}.  For $f \in L^2(\mathcal{S}^m;\mathbb{R}^n)$, this null space consists of potential fields of the form $du$ for any $u$ that vanishes at infinity.  
 
The structure of this null space, particularly in the case of bounded support, is important for reconstruction.  In this context, we will explore the mechanics of linear elastic systems in the context of tensor decompositions and inversion formulas related to the LRT. Through this process, we will provide two direct inversion techniques for the LRT for two-dimensional elastic strain fields that satisfy mechanical equilibrium on bounded domains in the absence of externally applied traction forces.  While the detailed context and precise definitions will follow, along the way we will demonstrate;
\begin{enumerate}
    \item{In the case of two-dimensional elastic strain fields, the assumption of zero boundary traction, a condition on stress, implies that the Helmholtz decomposition of strain on a bounded and unbounded domain are equivalent (up to extension by zero).   We also demonstrate more generally in $\mathbb{R}^n$ that the Helmholtz decomposition of a symmetric tensor field on a bounded or unbounded domain are equivalent if and only if the harmonic component of the Helmholtz decomposition on the bounded domain is zero (Lemma \ref{lem_ext}). }
    \item{The general inversion formula for the solenoidal component of the LRT on $m$-rank tensor fields over the entirety of $\mathbb{R}^n$ due to Sharafutdinov \cite{sharafutdinov2012integral} is equivalent to that of Louis \cite{louis2022inversion} and Derevtsov \textit{et al.} \cite{derevtsov2015tomography} for $m=n=2$.  The latter inversion formula was previously restricted to fields with zero harmonic component on the unit ball -- we extend its use to all of $\mathbb{R}^2$ (Lemma \ref{lem1}).}
    \item{In $\mathbb{R}^2$, the Helmholtz decomposition of any elastic strain field can be specified directly through the concept of an Airy stress potential and Hooke's law (Proposition \ref{SupportLemma}).}
    \item{The application of standard scalar filtered back projection to the LRT recovers the trace of the solenoidal component, which in the case of elastic strain in $\mathbb{R}^2$ is proportional to the hydrostatic component of stress (see section \ref{sec6}).}
\end{enumerate}
 
We begin by introducing the notation used throughout the paper.

 \section{Notation and definitions}

First, $\Omega$ will be an open subset of $\mathbb{R}^n$ with Lipschitz boundary possibly equal to $\mathbb{R}^n$ in the following definitions, and we will write $\mathbb{S}^{n-1}$ for the unit sphere in $\mathbb{R}^n$. Given a vector $v \in 
\mathbb{R}^2$, we write $v^\perp$ for the anti-clockwise rotation of $v$ by 90 degrees. If $v=(v_1,v_2)$, $v^\perp=(-v_2,v_1)$. The set $\mathcal{C}^\infty(\mathcal{S}^m;\Omega)$ will be the space of smooth $m$-rank symmetric tensor fields on $\Omega$ with continuous derivatives of all orders and $\mathcal{C}_c^\infty(\mathcal{S}^m;\Omega)$ the subspace of $\mathcal{C}^\infty(\mathcal{S}^m;\Omega)$ comprising fields with compact support in $\Omega$.

We use the following differential operators:
\begin{itemize}
	\item[]{$d$ -- Symmetric gradient operator. For $f \in \mathcal{C}^\infty(\mathcal{S}^m;\mathbb{R}^n)$, $df \in \mathcal{C}^\infty(\mathcal{S}^{m+1};\mathbb{R}^n)$ will be the symmetric derivative defined in \cite{sharafutdinov2012integral}. This coincides with the gradient when $m = 0$ and for $u \in \mathcal{C}^\infty(\mathcal{S}^1;\mathbb{R}^n)$
    \[
    (du)_{ij} = \frac{1}{2}\left (\frac{\partial u_{i}}{\partial x_{j}} + \frac{\partial u_{j}}{\partial x_{i}} \right ),
    \]
    or equivalently $du=\tfrac{1}{2}\big(\nabla \otimes u + (\nabla \otimes u)^T\big)$, where $\otimes$ refers to dyadic product and $(\cdot)^T$ refers to the transpose operation;}
    \item[]{$d^\perp$ -- Perpendicular symmetric gradient operator. Note this operator is only defined in dimension $n=2$. For $f \in \mathcal{C}^\infty(\mathcal{S}^m;\mathbb{R}^2)$, $d^{\perp}f \in \mathcal{C}^\infty(\mathcal{S}^{m+1};\mathbb{R}^2)$ is the symmetrisation of the perpendicular gradient of the components of $f$ introduced in \cite{derevtsov2015tomography}. 

    For $\psi \in \mathcal{C}^\infty(\mathbb{R}^2)$ this is given by
    \[
    (d^\perp \psi)_{i} = \frac{\partial \psi}{\partial x^j} e_{ji3}
    \]
    and for $u \in \mathcal{C}^\infty(\mathcal{S}^1;\mathbb{R}^2)$
    \[
    (d^\perp u)_{ij} = \frac{1}{2}\left (\frac{\partial u_{i}}{\partial x_k} e_{kj3} + \frac{\partial u_{j}}{\partial x_k} e_{ki3} \right ),
    \]
    where $e_{ijk}$ is the usual Levi-Civita permutation symbol.  Equivalently $d^\perp\psi = \nabla^\perp \psi$ and $d^\perp u=\tfrac{1}{2}\big(\nabla^\perp \otimes u + (\nabla^\perp \otimes u)^T\big)$;}
    
    \item[]
    {$\text{Div}$ -- The divergence operator which is the formal adjoint of $-d$ and maps $\mathcal{C}^\infty(\mathcal{S}^{m+1};\mathbb{R}^n) \rightarrow \mathcal{C}^\infty(\mathcal{S}^{m};\mathbb{R}^n)$. This is the contraction of the gradient of a tensor field and for the general formula see \cite{sharafutdinov2012integral}. For $u \in \mathcal{C}^\infty(\mathcal{S}^1;\mathbb{R}^n)$, $\text{Div}(u)$ is the standard divergence of $u$;}
    \item[]
    {$\text{Div}^\perp$ -- The perpendicular divergence which is the formal adjoint of $-d^\perp$ and maps $\mathcal{C}^\infty(\mathcal{S}^{m+1};\mathbb{R}^2) \rightarrow \mathcal{C}^\infty(\mathcal{S}^{m};\mathbb{R}^2)$. This is the same as the operator $\delta^\perp$ in \cite{derevtsov2015tomography}.}
    
\end{itemize}
We additionally say that a tensor field is divergence-free or solenoidal if its divergence is zero. The differential operators are initially defined on smooth tensor fields, but can be extended to fields with distributional coefficients.

For function spaces we use:

 \begin{itemize}
	\item[]{$L^2(\mathcal{S}^m;\Omega)$ -- The space of square-integrable $m$-rank symmetric tensor fields on $\Omega$ with norm $\|u\|_{L^2(\mathcal{S}^m;\Omega)}$}.
        \item[]
    {$H^k(\mathcal{S}^m;\Omega)$ -- The Sobolev space of square-integrable $m$-rank symmetric tensor fields on $\Omega$ whose weak derivatives up to order $k$ are also square-integrable. }
	
        \item[] 
        
    $\dot{H}^1_0(\mathcal{S}^m;\Omega)$ -- The homogeneous Sobolev space which is the closure of $\mathcal{C}_c^\infty(\mathcal{S}^m;\Omega)$ with respect to the norm $\|u\|_{\dot{H}^1_0(\mathcal{S}^m;\Omega)}=\|d u\|_{L^2(\mathcal{S}^{m+1};\Omega)}$.
    
        \item[]
    
    $\dot{H}^2_0(\mathbb{R}^n)$ --
    The homogeneous Sobolev space which is the closure of $\mathcal{C}_c^\infty(\Omega)$ with respect to the norm $\| \Delta u\|_{L^2(\mathbb{R}^n)}$.
    \end{itemize}
The homogeneous Sobolev spaces are equivalent to the standard Sobolev spaces of fields with trace zero when $\Omega$ is bounded, but different for unbounded $\Omega$.
    
We will mostly be concerned with tensors of rank either $m=1$ or $2$ and use the standard notations $f:g$ for contraction of 2-rank tensors and $f\cdot g$ for multiplication of a 2-rank tensor with a 1-rank tensor, or the dot product of 1-rank tensors.

We now return to the topic and begin with a review of Helmholtz decomposition and inversion of the LRT, both in general, and in the context of elastic strain in $\mathbb{R}^2$.

\section{Helmholtz decompositions and LRT inversion formulas}


As per \cite{sharafutdinov2012integral} and others, the null space of the LRT forms part of the orthogonal Helmholtz decomposition in $\mathbb{R}^n$ of symmetric tensor fields of the form
\begin{equation}
\label{decomp}
f=du+{^s}f,
\end{equation}
where ${^s}f$ is the divergence-free `solenoidal' component of $f \in L^2(\mathcal{S}^m;\mathbb{R}^n)$, $m\geq 1$, and $u \in \dot{H}_0^1(\mathcal{S}^{m-1};\mathbb{R}^n)$ gives the `potential' part $du$. Here the differential operators are understood to act in the sense of distributions on $\mathbb{R}^n$, and for given $f$ the decomposition \eqref{decomp} is unique. 

Using the fundamental theorem of calculus, it is easy to check that $If = I{^s}f$, and so at best we can hope to recover ${^s}f$ from the LRT of $f$. In fact, such recovery is possible as demonstrated by Sharafutdinov \cite{sharafutdinov2012integral} (see \eqref{fullSharafutdinov} below). However, an interesting practical problem exists when applying this to real systems; even if $f$ is compactly supported, $u$ and ${^s}f$ in \eqref{decomp} may have unbounded support, and for practical computation it is usually necessary to consider a bounded domain. In this light, let us introduce solenoidal decomposition on a bounded domain.

Let $\Omega \subset \mathbb{R}^n$ be a bounded domain with Lipschitz boundary and outward surface normal $n$ on $\partial\Omega$. Similar to (\ref{decomp}), there is a unique decomposition of $f \in L^2(\mathcal{S}^m;\mathbb{R}^n)$, $m \geq 1$, restricted to this set of the form (see \cite{schweizer} for the case of vector fields)
\begin{equation}\label{decomp_bd}
f = du_{\Omega} + {^s}f_\Omega + dh_{\Omega} \quad \mbox{on $\Omega$},
\end{equation}
where $u_\Omega \in \dot{H}^1_0(\mathcal{S}^{m-1};\Omega)$, $h_{\Omega} \in H^1(\mathcal{S}^{m-1};\Omega)$, known as the `harmonic part', satisfies
\[
\text{Div} (d h_\Omega) = 0 \quad \mbox{on $\Omega$},
\]
and ${^s}f_{\Omega} \in L^2(\mathcal{S}^m;\Omega)$ satisfies the weak equation
\begin{equation}\label{weakfomega}
\int_\Omega ({^s}f_{\Omega})_{i_1 \ ... \ i_m}  \frac{\partial \varphi_{i_1 \ ... \ i_{m-1}}}{{\partial x^{i_m}}} \ \mathrm{d} x = 0 \quad \forall \varphi \in H^1(\mathcal{S}^{m-1};\Omega).
\end{equation}
It is clear from \eqref{weakfomega} that ${^s}f_{\Omega}$ extended by zero to $\mathbb{R}^n$ is divergence-free. A key point for our result is that, for fields where the boundary trace makes sense, this extension by zero is only divergence-free when the boundary condition  $({^s}f_{\Omega})_{i_1 \ ... \ i_{m}}n_{i_{m}}|_{\partial \Omega} = 0$ holds. For an in depth discussion of weak formulation of the Helmholtz decomposition in the case of $L^2$ vector fields, see \cite{schweizer}.

To relate reconstruction formulae for the LRT on $\mathbb{R}^n$ to formulae on a bounded set, we must consider the relationship between the decompositions \eqref{decomp} and \eqref{decomp_bd}. Indeed, when the harmonic part vanishes in \eqref{decomp_bd}, the solenoidal decomposition on the bounded set $\Omega$ is related to the one on $\mathbb{R}^n$ as in the following lemma.

\begin{lemma} \label{lem_ext}
    Suppose that $\Omega$ contains the support of $f$. If $h_\Omega$ in \eqref{decomp_bd} is zero, then ${^s}f$ and $u$ in \eqref{decomp} are equal to the extension by zero of ${^s}f_\Omega$ and $u_\Omega$ to $\mathbb{R}^n$. Conversely, if ${^s}f$ and $u$ in \eqref{decomp} are supported in $\Omega$, then $h_\Omega = 0$.
\end{lemma}
\begin{proof}
    Assume that decomposition \eqref{decomp_bd} holds with $h_\Omega = 0$ and extend $u_\Omega$ and ${^s}f_\Omega$ to $v$ and $g$ on $\mathbb{R}^n$ by setting them equal to zero outside of $\Omega$. By \eqref{weakfomega}, $g$ is then divergence-free on $\mathbb{R}^n$. And since $u_\Omega \in \dot{H}^1_0(\mathcal{S}^{m-1};\Omega)$ then $v \in \dot{H}_0^1(\mathcal{S}^{m-1};\mathbb{R}^n)$ and $d v$ is $d u _\Omega$ extended by zero to $\mathbb{R}^n$. By uniqueness of the decomposition in \eqref{decomp}, $u=v$ and ${^s}f=g$.

    Conversely, suppose that ${^s}f$ and $u$ in \eqref{decomp} are supported in $\Omega$ and define $\tilde u_\Omega$ and ${^s}\tilde f_\Omega$ by restricting their domain to $\Omega$.  Then, since $u\in \dot{H}_0^1(\mathcal{S}^1;\mathbb{R}^n)$ with support contained in $\Omega$, its restriction $\tilde u_\Omega$ is in $\dot{H}^1_0(\mathcal{S}^1;\Omega)$. Additionally, we can see that \eqref{weakfomega} holds for ${^s}\tilde f_\Omega$ because the same must hold for ${^s}f$ on $\mathbb{R}^n$ for any $\varphi \in H^1(\mathcal{S}^{m-1};\mathbb{R}^n)$. By uniqueness of the decomposition  we see that \eqref{decomp_bd}  holds with $u_\Omega=\tilde u_\Omega$, ${^s}f_\Omega={^s}\tilde f_\Omega$ and $h_\Omega = 0$ on $\Omega$, as claimed.
\end{proof}

Now let us turn to inversion of the LRT. Various inversion formulas exist that can uniquely recover $^sf$ from $If$ (e.g. \cite{sharafutdinov2012integral, derevtsov2015tomography, louis2022inversion}).  Sharafutdinov \cite{sharafutdinov2012integral} provides the general result for $f \in L^2(\mathcal{S}^m;\mathbb{R}^n)$ as
\begin{equation}
\label{fullSharafutdinov}
^sf=(-\Delta)^{1/2}\Big[\sum_{k=0}^{[m/2]}c_k(i-\Delta^{-1}d^2)^kj^k\Big]\mu^m If,
\end{equation}
where $c_k$ are specified scalar coefficients, powers of the Laplacian $(-\Delta)^{1/2}$ and $(-\Delta)^{-1}$ are defined via the Fourier transform, the operators $i$ and $j$ respectively refer to product and contraction with the Kronecker tensor, and $\mu^m$ is the formal adjoint of $I$ when the measure on $\mathbb{S}^{n-1}$ is normalised to one.  In practical terms, $\mu^m$ is related to the adjoint of the X-ray transform\footnote{Equivalent to the Radon transform in 2D.} (i.e. scalar back-projection), $\mathcal{R^*}$, acting component-wise with back-projections weighted by the diadic product of $\xi$ with itself $m$-times;
\begin{equation}
\mu^m_{i_1i_2...i_m}= \frac{1}{2 \pi^{n/2}}\Gamma\left (\tfrac{n}{2} \right )\mathcal{R}^*\xi_{i_1}\xi_{i_2}...\xi_{i_m}.
\end{equation}
Note that the constant factor is present because of the normalisation of the measure on $\mathbb{S}^{n-1}$ in \cite{sharafutdinov2012integral}. 

For 2D elastic strain $\epsilon \in L^2(\mathcal{S}^2;\mathbb{R}^2)$, \eqref{fullSharafutdinov} simplifies to
\begin{equation}
{^s}\epsilon = \frac{1}{2\pi}(- \Delta)^{1/2}\Big[c_0 + c_1 (\text{\bf{I}}-\Delta^{-1}d^2) tr \Big] I^* I \epsilon,
\label{SharInv}
\end{equation}
where $c_0 = 3/4, c_1 = -1/4$, $tr$ is the trace operator, $\text{\bf{I}}$ is the 2-rank identity and $I^*=\mathcal{R}^*\xi\otimes\xi$.  
In comparison, Derevtsov and Svetov \cite{derevtsov2015tomography} and Louis \cite{louis2022inversion} consider recovery when $\Omega$ is the unit ball in $\mathbb{R}^2$, implicitly assuming also that the harmonic part of the field is equal to zero so that ${^s}\epsilon = {^s}\epsilon_\Omega$ by Lemma \ref{lem_ext}. In this context, \cite{louis2022inversion} provides a much simpler inversion formula of the form
\begin{equation}
^s\epsilon = 
\frac{1}{4 \pi}(- \Delta)^{1/2} I^* I \epsilon, \quad \textnormal{ in $\Omega$,}
\label{InvLRT}
\end{equation}
while Derevtsov and Svetov \cite{derevtsov2015tomography} provide the same formula \eqref{InvLRT} but, due to a typographical error, multiplied by a factor of $2$ on the right side. 

We now show in Lemma \ref{lem1} that \eqref{SharInv} and \eqref{InvLRT} are indeed equivalent. This extends the inversion results of \cite{louis2022inversion,derevtsov2015tomography} from the unit ball to $\mathbb{R}^2$, and handles the case of non-vanishing harmonic part, which was not considered in \cite{louis2022inversion,derevtsov2015tomography}.

\begin{lemma} \label{lem1}
    For any $\epsilon \in L^2(\mathcal{S}^2;\mathbb{R}^2)$, the right hand sides of \eqref{SharInv} and \eqref{InvLRT} are equal and hence \eqref{InvLRT} can be used on all of $\mathbb{R}^2$ regardless of any harmonic component.
\end{lemma}
\begin{proof}
Taking the component-wise Fourier transform with spatial frequency vector $\kappa$, (\ref{SharInv}) can be written
\begin{equation}
{^s}\hat\epsilon = \frac{1}{2\pi}|\kappa|\Big[c_0 + c_1 \Big(\text{\bf{I}}- \frac{\kappa\kappa^T}{|\kappa|^2} \Big)tr \Big] \hat{g},
\label{FSharInv}
\end{equation}
where $g = I^* I \epsilon$.  Since $^s\epsilon$ is solenoidal ${^s}\hat \epsilon \kappa =0$ and we can write ${^s}\hat \epsilon=\alpha \kappa^\bot (\kappa^\bot)^T$ for some $\alpha \in L^2(\mathbb{R}^2)$.  Hence (\ref{FSharInv}) becomes
\[
\alpha \kappa^\bot (\kappa^\bot)^T = \frac{1}{2\pi}|\kappa|\Big[c_0 + c_1 \Big(\text{\bf{I}}- \frac{\kappa\kappa^T}{|\kappa|^2} \Big)tr \Big] \hat{g}.
\]
Multiplying by $\kappa^\bot (\kappa^\bot)^T$ and rearranging;
\[
\alpha \kappa^\bot (\kappa^\bot)^T |\kappa|^2 = {^s}\hat\epsilon |\kappa|^2 = \frac{1}{2\pi}\kappa^\bot (\kappa^\bot)^T |\kappa|\Big[c_0 + c_1 \Big(\text{\bf{I}}- \frac{\kappa\kappa^T}{|\kappa|^2} \Big)tr \Big]  \hat g 
\]
which provides
\[
{^s}\hat\epsilon = \frac{1}{2\pi}\frac{\kappa^\bot (\kappa^\bot)^T}{|\kappa|}\Big[c_0 + c_1 \text{\bf{I}} tr \Big] \hat{g}.
\]
Now $g$ is also solenoidal and hence can also be written $\hat g=\beta \kappa^\bot (\kappa^\bot)^T$ for some $\beta \in L^2(\mathbb{R}^2)$;
\begin{align*}
{^s}\hat\epsilon &= \frac{1}{2\pi}\frac{\kappa^\bot (\kappa^\bot)^T}{|\kappa|}\Big[c_0 \kappa^\bot (\kappa^\bot)^T + c_1 \text{\bf{I}}|\kappa|^2 \Big] \beta \\
&=\frac{1}{2\pi}c_0 |\kappa|  \beta \kappa^\bot (\kappa^\bot)^T +\frac{1}{2\pi} c_1 |\kappa|  \beta \kappa^\bot (\kappa^\bot)^T \\
&=\frac{1}{2\pi}|\kappa|(c_0+c_1)\hat g.
\end{align*}
In the spatial domain this implies over all of $\mathbb{R}^2$:
\begin{align*}
{^s}\epsilon &= \frac{1}{2\pi}(-\Delta)^{1/2}(c_0+c_1)I^*I\epsilon \\
&=\frac{1}{4\pi}(-\Delta)^{1/2}I^*I\epsilon,
\end{align*}
which is identical to (\ref{InvLRT}) but on all of $\mathbb{R}^2$.
\end{proof}

Given Lemma \ref{lem1}, we use only \eqref{InvLRT} which provides a component-wise approach to reconstruction of the solenoidal component of strain in $\mathbb{R}^2$ of the form
\begin{equation} \label{TFBP}
{^s}\epsilon = \frac{1}{4\pi}\mathcal{R}^* \Lambda \xi \otimes \xi I\epsilon,
\end{equation}
where $\Lambda$ is the Ram-Lak filter (or similar) used in standard scalar Filtered Back Projection (FBP). 

Because of Lemma \ref{lem1}, we know that this inversion formula recovers the solenoidal part on all of $\mathbb{R}^2$ with potentially unbounded support regardless of the finite nature of the sample. By Lemma \ref{lem_ext}, the solenoidal component of $\epsilon$ will have support contained in a bounded domain only if its harmonic part vanishes, and so it is important to know when this will occur in the context of strain. 

Before we address this, we first provide a brief review of the mechanics of stress and strain on the plane in the context of this work.

\section{Elasticity theory and residual stress} \label{sec:elast}

Consider a sample consisting of an elastic body in $\mathbb{R}^3$ represented by the bounded domain $\Omega$ with outward surface normal $n$. Within $\Omega$ we can decompose the total strain at each point, $\epsilon_T$, into an elastic component, $\epsilon$ and an `eigenstrain', $\epsilon^*$ (e.g. permanent strain introduced by plasticity, phase change, thermal expansion, etc.) \cite{korsunsky2017teaching, mura1982micromechanics}
\begin{equation}
\epsilon_T=\epsilon+\epsilon^*.
\label{strain_decomp}
\end{equation}
The elastic component of strain is related to stress, $\sigma$, through Hooke's law, which in its most general form, can be written in terms of a 4-rank stiffness tensor; $\sigma_{ij}=C_{ijkl}\epsilon_{kl}$.  In the isotropic case with Young's modulus $E$ and Poisson's ratio $\nu$
\begin{equation}
    \label{Hooke}
    C_{ijkl}= \frac{E}{1+\nu} \Big(\frac{\nu}{1-2\nu}\delta_{ij} \delta_{kl} + \frac{1}{2}\left( \delta_{ik} \delta_{jl} + \delta_{il} \delta_{jk}\right)\Big).
\end{equation}
Governing equations can be assembled for this system on the basis of equilibrium, compatibility of strain and boundary conditions.  In the absence of body forces (gravity, magnetism, etc.) mechanical equilibrium holds that
\begin{equation}
\label{equilib}
\text{Div}(\sigma) = \text{Div}\big(C:\epsilon\big)=0.
\end{equation}
The total strain physically originates as the symmetric gradient of a displacement field (i.e. is potential) and can be expressed as $\epsilon_T = du$ for some $u$, where, in general, $u\ne 0$ on $\partial\Omega$.  This condition is known as strain `compatibility' which for a simply connected domain can be expressed as a vanishing Saint-Venant 
operator\footnote{The Saint-Venant operator is defined by
\[
W_{ijkl}(f)=\frac{\partial^2f_{ij}}{\partial x_k\partial x_l}  + \frac{\partial^2f_{kl}}{\partial x_i\partial x_j}  -\frac{\partial^2f_{il}}{\partial x_j\partial x_k} - \frac{\partial^2f_{jk}}{\partial x_i\partial x_l}.
\]
In $\mathbb{R}^3$, this simplifies to six unique components specified by the 2-rank symmetric incompatibility tensor $Rf=\nabla \times (\nabla \times f)^T$, or component-wise $[Rf]_{ij}=e_{kpi}e_{lqj}\nabla_p\nabla_q f_{kl}$ where $e_{ijk}$ is the Levi-Civita permutation symbol.  

In a simply connected domain in $\mathbb{R}^n$, $W(f)=0$ if and only if $f=du$ for some $u$. On a multiply connected domain with $k$ holes, $n(n+1)k/2$ additional integral constraints are required along with $W(f)=0$ to imply $f=du$ (see \cite[Proposition 2.8]{yavari2013compatibility}).}, $W(\epsilon_T)=0$, or
\begin{equation}
\label{compat}
W(\epsilon)=-W(\epsilon^*).
\end{equation}
The final ingredient is to specify boundary conditions experienced by the sample.  These can vary, but in the case of `residual stress' problems, the surface of the sample is typically free of any traction
\begin{equation}
\label{BC}
\sigma \cdot n = \big(C:\epsilon\big) \cdot n=0 \text{ on } \partial\Omega.
\end{equation}
Equations (\ref{equilib}), (\ref{compat}) and (\ref{BC}) together form an elliptic boundary value problem for $\epsilon$ based on a known eigen-strain $\epsilon^*$.

While $\sigma$ and $\epsilon$ are inherently three-dimensional in nature, there are two typical limiting assumptions on the plane that have practical utility \cite{Timoshenko1970Theory}:
\begin{enumerate}
    \item{Plane-strain conditions ($\epsilon_{i3}=0 \quad \forall i)$};
    \item{Plane-stress conditions ($\sigma_{i3}=0 \quad \forall i$)}.
\end{enumerate}
Plane-strain is a limiting case for thick prismatic samples, while plane-stress relates to thin two-dimensional samples where `thick' and `thin' refer to dimensions in the $x_3$ direction. The above analysis applies directly to both cases where $\Omega \in \mathbb{R}^2$ with the exception that, in the plane-stress case, the isotropic elasticity tensor becomes
\begin{equation}
    \label{HookePlaneStress}
    C_{ijkl}= \frac{E}{1+\nu} \Big(\frac{\nu}{1-\nu}\delta_{ij} \delta_{kl} + \frac{1}{2}\left( \delta_{ik} \delta_{jl} + \delta_{il} \delta_{jk}\right)\Big).
\end{equation}

\section{Problem statement}
We are now in a position to state precisely the inverse problem we seek to solve in this work.
\begin{problem}
    \label{theproblem}
    Given $\Omega \subset \mathbb{R}^2$ and $\epsilon \in L^2(\mathcal{S}^2;\Omega)$, where $\epsilon$ derives from a plane-stress or plane-strain state and is known to satisfy \eqref{equilib} and \eqref{BC}, we wish to recover $\epsilon$ from its LRT.
\end{problem}
The rest of the paper is focused on developing a solution to this problem and demonstrating its numerical implementation.

\section{Helmholtz decomposition of strain in $\mathbb{R}^2$}

We begin by connecting the stress $\sigma$ and strain $\epsilon$ initially defined only on the bounded set $\Omega$ to the solenoidal decomposition \eqref{decomp} on all of $\mathbb{R}^2$. Given that the stress $\sigma$ satisfies \eqref{equilib} in the classical sense (i.e. is twice differentiable) on $\Omega$ and satisfies the traction-free boundary condition \eqref{BC}, in fact $\sigma$ extended as zero outside of $\Omega$ is divergence free in the distributional sense and is therefore its own solenoidal part with no potential part if decomposed according to \eqref{decomp}. Our goal in this section is to use this fact, together with \eqref{planestess} or \eqref{planestrain} to find the solenoidal decomposition of $\epsilon$. 

This can be achieved through the concept of an Airy stress function.  In both the plane-stress and plane-strain cases, it is possible to write $\sigma$ in terms of a scalar Airy stress potential, $\psi \in H^2(\Omega)$ in such a way that it automatically satisfies equilibrium:
\begin{equation} \label{Airy1}
    \sigma=(d^\perp)^2\psi.
\end{equation}
When combined with Hooke's law (i.e. \eqref{Hooke} or \eqref{HookePlaneStress}), it follows that strain can also be written in terms of this same potential as
\begin{equation}
    \label{planestrain}
    \epsilon=\frac{1+\nu}{E}\Big((1-\nu)(d^\perp)^2-\nu d^2\Big) \psi
\end{equation}
for plane-strain conditions, or
\begin{equation}
    \label{planestess}
    \epsilon=\frac{1}{E}\Big((d^\perp)^2-\nu d^2\Big) \psi
\end{equation}
in the case of plane-stress.

Both \eqref{planestess} and \eqref{planestrain} already appear to be in the form of Helmholtz decompositions, however the issue is that the Airy stress potential appearing in \eqref{Airy1} may not satisfy equilibrium in a distributional sense when extended as zero to $\mathbb{R}^2$. The next lemma shows that when the traction-free boundary condition \eqref{BC} is satisfied, in fact there is an Airy stress potential which extends as zero.

\begin{proposition}
\label{SupportLemma}
    Suppose that $\sigma \in L^2(\mathcal{S}^2;\mathbb{R}^2)$ has support contained in a bounded and simply connected set $\Omega$ and satisfies \eqref{equilib} in the distributional sense on $\mathbb{R}^2$. Then there exists unique $\psi \in \dot{H}_0^2(\mathbb{R}^2)$ such that $\mathrm{supp}(\psi) \subset \Omega$ and
    \begin{equation}\label{sigma_pot}
    \sigma = (d^\perp)^2 \psi \quad \mbox{on $\mathbb{R}^2$.}
    \end{equation}
    Furthermore,
    \begin{equation}\label{psi_cont}
    \|\psi\|_{\dot{H}_0^2(\mathbb{R}^2)} \leq M \|\sigma\|_{L^2(\mathcal{S}^2;\mathbb{R}^2)}.
    \end{equation}
    for a constant $M>0$ which depends on $\Omega$ but not $\sigma$.
\end{proposition}

\begin{proof}
    First consider the case when $\sigma \in \mathcal{C}_c^\infty(\mathcal{S}^2;\mathbb{R}^2)$ satisfies \eqref{equilib} and has support contained in $\Omega$ which is itself inside an open ball $B_R$ of radius $R$ centred at the origin. The two columns of $\sigma$, $\sigma_{i1}$ and $\sigma_{i2}$, are divergence free vector fields on $\mathbb{R}^2$ and so the path integrals of $e_{ik3} \sigma_{ij} dx_k$ between any two points are independent of path due to Green's theorem. For $x_0 \in \partial B_R$ and any $x \in \mathbb{R}^2$, we define new functions via the path integrals
    \begin{equation}\label{path1}
    \phi_j(x) = \int_{x_0}^x e_{ik3} \sigma_{ij} dx_k
    \end{equation}
    in which the path is left unspecified. Defining the vector field $\phi = (\phi_1, \phi_2)$ it follows, due to path independence and the fundamental theorem of calculus, that
    \begin{equation} \label{dphi}
    \frac{\partial \phi_j}{\partial x_k} = e_{ik3} \sigma_{ij}.
    \end{equation}
    Additionally, since $\Omega$ is simply connected, for any $x \in \mathbb{R}^2 \setminus \Omega$ we can choose a path from $x_0$ to $x$ outside of $\Omega$ and by its path integral definition \eqref{path1}, we have $\phi(x) = 0$. Thus, we conclude that $\phi$ is also supported in $\Omega$.

    Next, from \eqref{dphi} we obtain
    \[
    \frac{\partial \phi_1}{\partial x_1} + \frac{\partial \phi_2}{\partial x_2} = \text{Div}(\phi) = 0.
    \]
    This implies as before that line integrals of $e_{ik3} \phi_i\ d x_k$ between two points are independent of path, and we define
    \[
    \psi(x) = \int_{x_0}^x e_{jl3} \phi_j \ d x_l.
    \]
    Also as before, this implies that $\psi$ is supported in $\Omega$ and
    \[
    \frac{\partial \psi}{\partial x_l} = e_{jl3} \phi_j.
    \]
    Putting together the previous construction and using path independence we see that $\psi$ is directly related to $\sigma$ by the formula
    \[
    \psi(x_1,x_2) = \iint_{\{s<x_1, \ t >x_2\}} \sigma_{12}(s,t) \ \mathrm{d}s \ \mathrm{d} t.
    \]
    Since the support $\sigma$ is bounded we can restrict the area of integration in the previous integrals to bounded rectangles, and then use the Cauchy-Schwartz inequality to prove \eqref{psi_cont} where the constant $M$ depends only on the size of $\Omega$.

    We have now proved the proposition for the case when $\sigma$ is smooth. For $\sigma \in L^2(\mathcal{S}^2;\mathbb{R}^2)$ we approximate by a sequence $\sigma_j \in \mathcal{C}_c^\infty(\mathcal{S}^2;\mathbb{R}^2)$ of divergence free fields such that $\sigma_j \rightarrow \sigma$ in $L^2(\mathcal{S}^2;\mathbb{R}^2)$ and each $\sigma_j$ is supported within a domain with its boundary within a distance of $2^{-j}$ from $\partial\Omega$. By \eqref{psi_cont} the corresponding potentials $\psi_j$ also converge in $H^2(\mathbb{R}^2)$ to a function $\psi$ and by continuity of the derivatives from $H^2$ to $L^2$ we see that \eqref{sigma_pot} also holds. The supports of the potentials will also shrink to $\Omega$ and so we see that the support of $\psi$ is contained in $\Omega$.

    Finally, note that from \eqref{sigma_pot} the potential $\psi \in H^2(\mathbb{R}^2)$ satisfies the biharmonic equation
    \[
    \Delta^2 \psi = (\text{Div}^\perp)^2 \sigma.
    \]
    This equation has a unique solution in $H^2(\mathbb{R}^2)$ and so the proof is complete.
\end{proof}

\noindent
From Lemma \ref{SupportLemma}, we can conclude the following:

If a two dimensional residual elastic strain field on the simply connected bounded domain $\Omega$ exists in the absence of boundary traction, its extension by zero to all of $\mathbb{R}^2$ has a unique Helmholtz decomposition of the form
\begin{equation}
    \label{StrainDecomp}
    \epsilon=d\omega + {^s\epsilon}
\end{equation}
where ${^s\epsilon}$ and $d\omega$ are compactly supported within $\Omega$. Note that we only assume that the support of $\epsilon$ is contained within the simply connected set $
\Omega$, not that the support of $\epsilon$ is itself simply connected. By uniqueness and comparison to \eqref{planestrain} and \eqref{planestess}, this decomposition can be written in terms of the Airy stress potential as
\begin{align}
    \omega&=-\frac{\nu(1+\nu)}{E}d\psi \\
    \label{PStrainAsStress}
    ^s\epsilon&=\frac{1-\nu^2}{E}(d^\perp)^2\psi,
\end{align}
in the case of plane-strain, or
\begin{align}
    \omega&=-\frac{\nu}{E}d\psi \\
    \label{PStressAsStress}
    ^s\epsilon&=\frac{1}{E}(d^\perp)^2\psi,
\end{align}
for plane-stress.  Note that in each case $^s\epsilon$ is proportional to $\sigma$.

From this decomposition and the inversion formula for $^s\epsilon$ we now seek to recover the full elastic strain tensor over a sample.  Before we approach this task, we provide a brief comment on recent experimental work in this area.

\section{Isotropic strain and scalar Filtered Back Projection}
\label{sec6}

Some recent work in Bragg-edge strain tomography has approached this problem through an assumption that strain is isotropic at all points within the sample; i.e. $\epsilon=\bar\epsilon \hspace{0.3ex} \text{\bf{I}}$ for some scalar mean strain $\bar\epsilon$.  This assumption is plainly false in almost all cases; the only hydrostatic stress field (and hence strain field) that satisfies equilibrium is constant for all $x$.  However, the assumption does allow for a direct means of reconstruction by standard scalar FBP since $I\epsilon=\mathcal{R}\bar\epsilon$ for this case.

For example, in Busi \textit{et al} \cite{busi2022bragg} the authors perform a slice-by-slice FBP to recover an assumed isotropic strain within an additively manufactured stainless steel cube from a set of 19 Bragg-edge strain images.  Similarly, Zhu \textit{et al} \cite{zhu2023bragg} recover an assumed scalar isotropic strain in a laser welded steel sample using a similar technique.  

Clearly the assumption of isotropic strain was invalid in both cases, however the question remains: What has been recovered?  How does the scalar FBP of the LRT relate to the strain field within the sample?

To answer this question, we examine the trace of the solenoidal component of elastic strain in (\ref{TFBP}) to obtain the following (note that $|\xi|=1$);
\begin{align*}
{^s\epsilon}_{kk} 
	&=\frac{1}{4 \pi} \mathcal{R}^*\xi_k\xi_k \Lambda I \epsilon \\
	&= \frac{1}{4 \pi}\mathcal{R}^*\Lambda I \epsilon.
\end{align*}
Hence the recovered scalar field stemming from an isotropic assumption is precisely the trace of the (in-plane) solenoidal component, and in general there are no further conclusions that can be made.  

However, if the strain field is inherently two-dimensional, we can extend this result by considering stress in terms of the Airy potential. As before, under plane-stress or plane-strain conditions, $^s\epsilon$ can be interpreted through the natural Helmholtz decompositions (\ref{PStrainAsStress}) and (\ref{PStressAsStress}).  From this perspective, it follows that for plane-strain
\begin{equation}
    \frac{1}{4\pi}\mathcal{R}^*\Lambda I \epsilon=\frac{1-\nu^2}{E}\sigma_{kk},
\end{equation}
and for plane-stress
\begin{equation}
    \frac{1}{4\pi}\mathcal{R}^*\Lambda I \epsilon=\frac{1}{E}\sigma_{kk}.
\end{equation}

\section{Recovery of $\epsilon$ from $^s\epsilon$}

We now turn our attention to the problem of recovering $\epsilon$ from $^s\epsilon$ using the constraints provided by elasticity theory.  To this end, we present three approaches to the solution of Problem \ref{theproblem}.

\subsection{Recovery of $\epsilon$ from compatibility}

Applying the Saint-Venant operator to \eqref{StrainDecomp} implies $W(\epsilon)=-W(\epsilon^*)=W({^s}\epsilon)$ and we can replace the compatibility relation (\ref{compat}) to form a boundary value problem for $\epsilon$;
\begin{equation}
\label{equilibrium}
\begin{cases}
\text{Div}(C:\epsilon)=0 & \text{(Equilibrium)} \\
W(\epsilon)=W({^s\epsilon}) & \text{(Compatibility)}\\
(C:\epsilon) n =0 \text{ on } \partial\Omega &\text{(Boundary condition)}
\end{cases}
\end{equation}

Under two-dimensional plane-stress or plane-strain conditions we can satisfy equilibrium via \eqref{planestrain} or \eqref{planestess}, and the compatibility condition becomes a non-homogeneous bi-harmonic equation
\begin{equation}
\Delta^2 \psi = \frac{\partial^4\psi}{\partial x_1^4}+\frac{\partial^4\psi}{\partial x_2^4}+2\frac{\partial^4\psi}{\partial x_1^2 \partial x_2^2} = E(\nabla^\perp)^T {^s\epsilon} \nabla^\perp,
\label{biharmonic}
\end{equation}
subject to the boundary condition
\begin{equation}
    (d^\perp)^2\psi \cdot n = 0 \text{ on } \partial\Omega.
\end{equation}
Potentially this provides a direct approach to recover $\psi$ and hence $\epsilon$ through numerical solution.  However, it should be recognised that computing the right hand side of \eqref{biharmonic} involves taking second order numerical derivatives.  In the presence of experimental uncertainty, this is likely to be a very unstable process.

\subsection{Recovery of the potential component}

An alternate approach involves the recovery of the potential part of $\epsilon$ using equilibrium.  From (\ref{StrainDecomp}) and (\ref{equilibrium}), the equilibrium of the system implies
\begin{align}
\text{Div}\big(C:(d\omega+{^s\epsilon})\big)&=0,
\end{align}
which leads to an elliptic boundary value problem for $\omega$ of the form
\begin{align}
\label{OmegaEqu}
\text{Div}(C:d\omega)=b \\
\label{OmegaBC}
\omega = 0 \text{ on } \partial \Omega
\end{align}
where $b=-\text{Div}(C:{{^s}\epsilon})$.

This is in the form of a standard structural elasticity problem for $\omega$ as a displacement field resulting from a distributed body force and trivial Dirichlet boundary condition. For 2D plane-stress conditions
\begin{align}
\label{bx}
b_1&=-\frac{E}{1-\nu^2}\Big(\frac{\partial{{^s}\epsilon_{11}}}{\partial x_1} + \nu \frac{\partial{{^s}\epsilon_{22}}}{\partial x_1} + (1-\nu) \frac{\partial{{^s}\epsilon_{12}}}{\partial x_2}\Big), \\
b_2&=-\frac{E}{1-\nu^2}\Big(\nu\frac{\partial{{^s}\epsilon_{11}}}{\partial x_2} + \frac{\partial{{^s}\epsilon_{22}}}{\partial x_2} + (1-\nu) \frac{\partial{{^s}\epsilon_{12}}}{\partial x_1}\Big).
\label{by}
\end{align}
In contrast to the previous approach, calculation of $b$ only involves computing first derivatives, and hence is potentially a much more stable process.

\subsection{Recovery of $\epsilon$ from Hooke's law}

By far the most direct means for recovering $\epsilon$ from $^s\epsilon$ is through Hooke's law. From Proposition \ref{SupportLemma}, we have that $\sigma=(d^\perp)^2\psi$ for some Airy stress function $\psi$, and therefore using \eqref{PStrainAsStress} and \eqref{PStressAsStress} together with \eqref{planestrain} and \eqref{planestess}, we can write
\begin{align}
    \epsilon_{11}&={^s\epsilon}_{11}+\frac{\nu}{1-\nu}{^s\epsilon}_{22} \\ 
    \epsilon_{22}&={^s\epsilon}_{22}+\frac{\nu}{1-\nu}{^s\epsilon}_{11} \\
    \epsilon_{12}&=\frac{1}{1-\nu}{^s\epsilon}_{12}
\end{align}
for plane-strain, or
\begin{align}
    \label{HookeRecon1}
    \epsilon_{11}&={^s\epsilon}_{11}-\nu{^s\epsilon}_{22} \\ 
    \label{HookeRecon2}
    \epsilon_{22}&={^s\epsilon}_{22}-\nu{^s\epsilon}_{11} \\
    \label{HookeRecon3}
    \epsilon_{12}&=(1+\nu){^s\epsilon}_{12}
\end{align}
for plane-stress conditions.

\section{Numerical demonstration: Simulated data}

\subsection{Strain fields}

Numerical demonstrations of the above process were performed on three synthetic two-dimensional plane-stress strain fields. The first of these fields was generated over the unit disk from an Airy stress potential of the form
\begin{equation}
\psi=e^{-\alpha((x+1/4)^2+y^2)}-e^{-\alpha((x-1/4)^2+y^2)},
\end{equation}
with $\alpha=15$, and elastic properties $E=1$ and $\nu=0.34$.  The three independent components of this strain field are shown in Figure \ref{AiryField}a.

\begin{figure} [h!]
\begin{center}
    \includegraphics[width=0.24\linewidth]{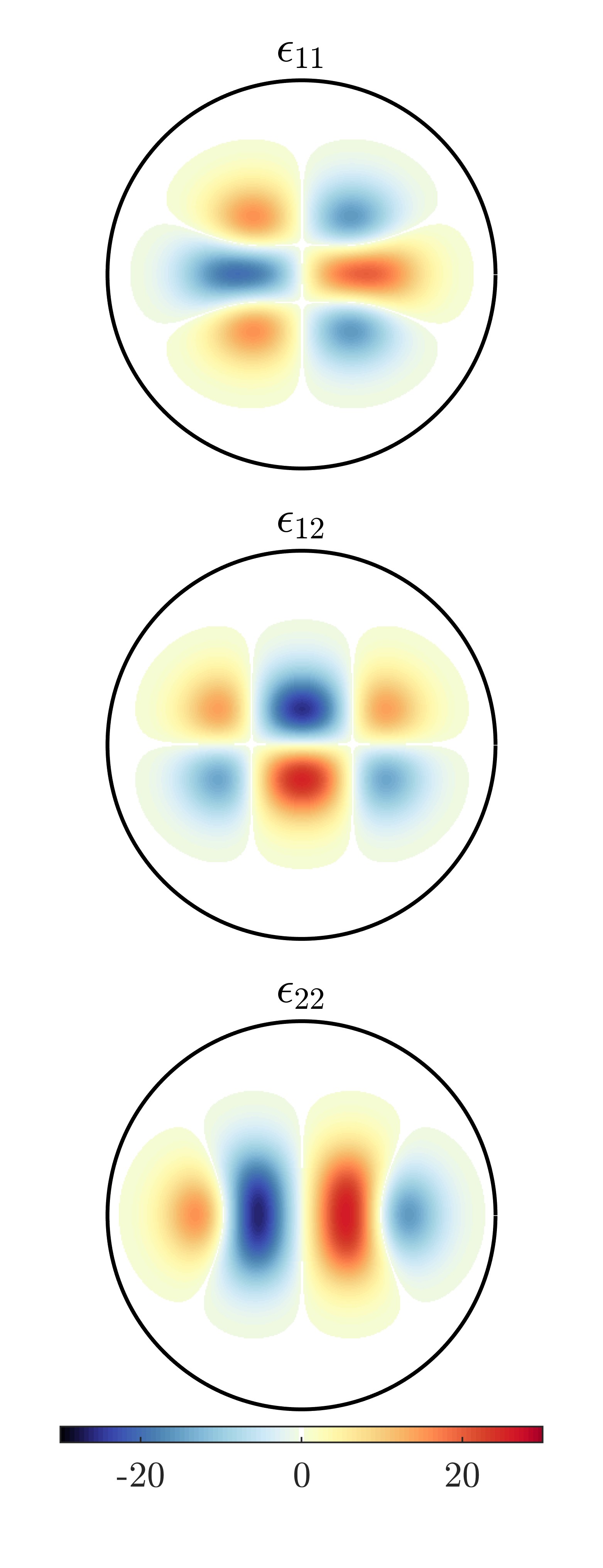}
    \includegraphics[width=0.24\linewidth]{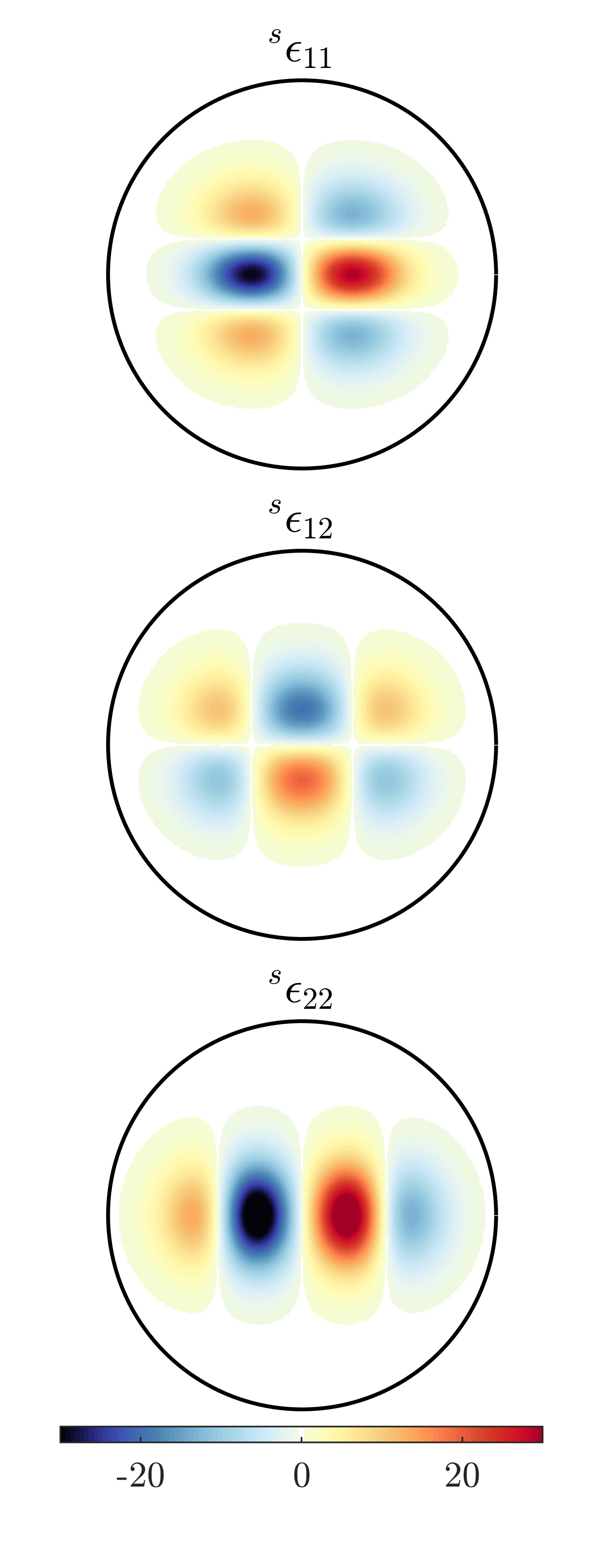}
    \includegraphics[width=0.24\linewidth]{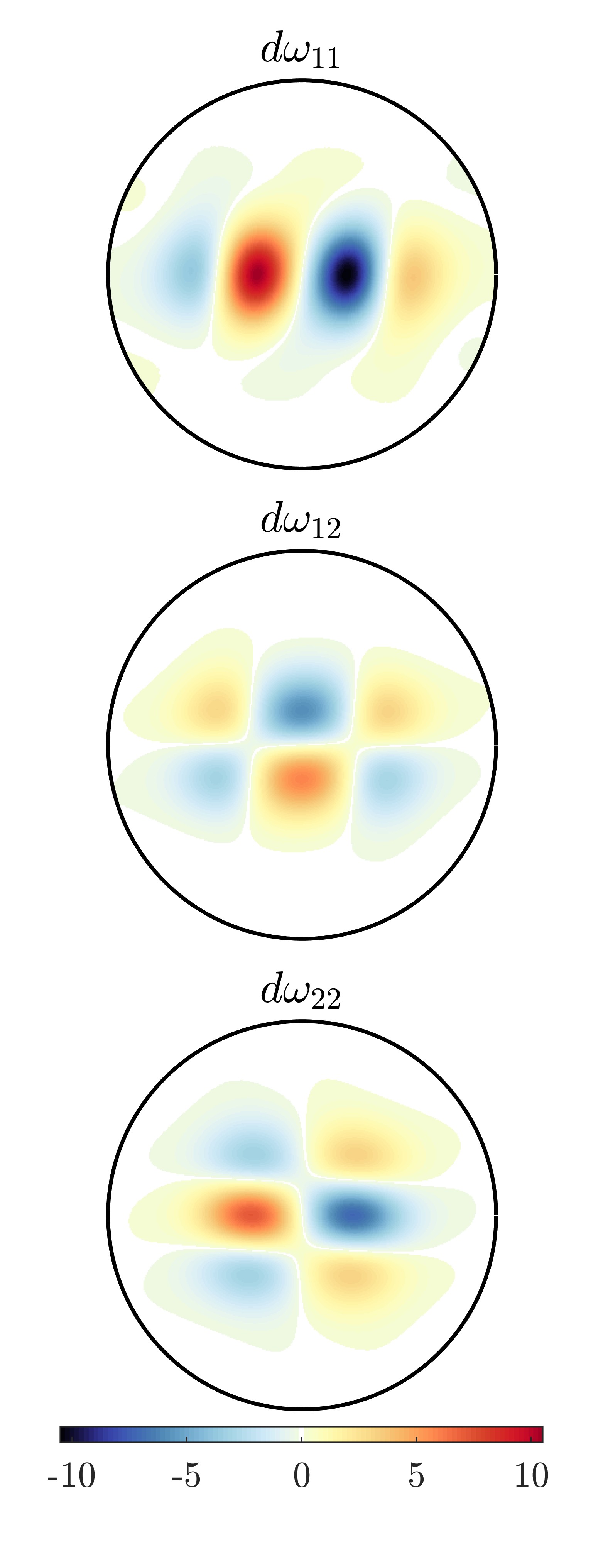}
    \includegraphics[width=0.24\linewidth]{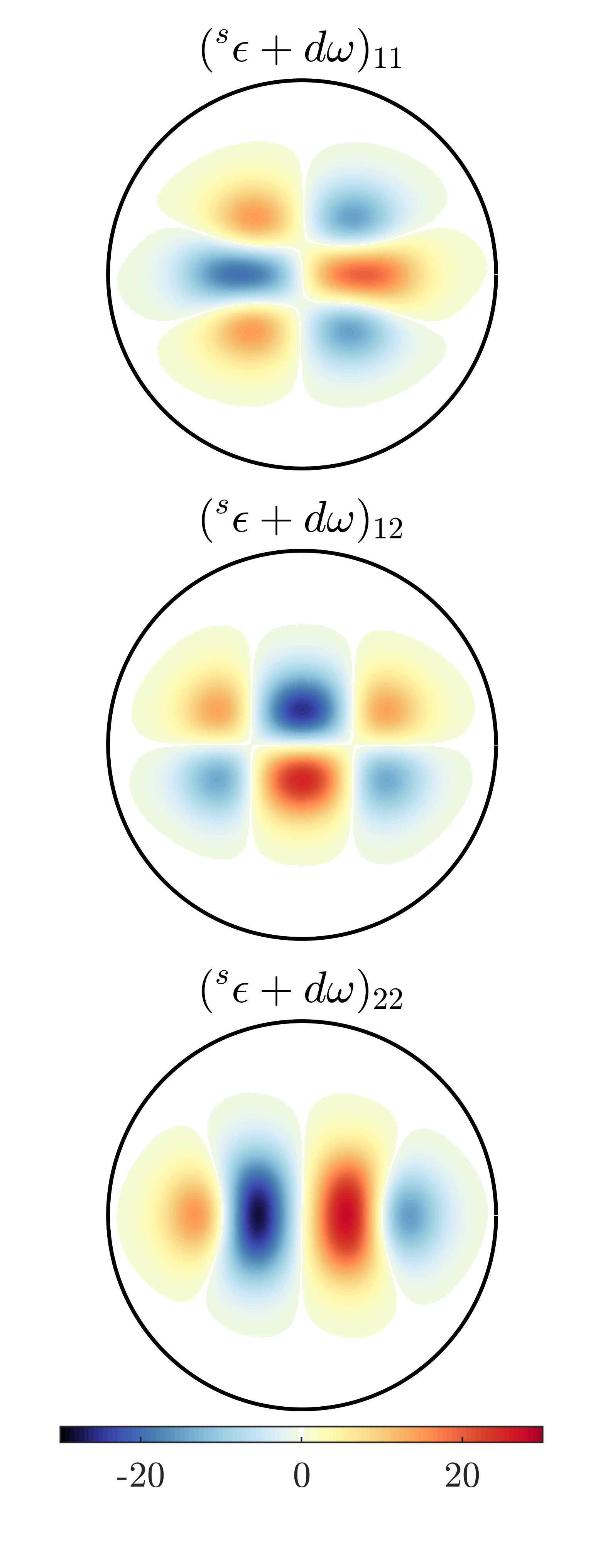}
    \put(-380,230){(a)}
    \put(-282.5,230){(b)}
    \put(-183.5,230){(c)}
    \put(-85,230){(d)}
    \caption{A reconstruction of a synthetic strain field computed from an Airy stress field. (a) The original strain field. (b) A reconstruction of the solenoidal component of this field from a simulated LRT consisting of 200 equally spaced projections over 360$^\circ$. (c) The recovered potential component from elastic finite element modelling. (d) The reconstructed strain field formed by the sum of the solenoidal and potential components.}
    \label{AiryField}
\end{center}
\end{figure} 

The second and third fields corresponded to finite element simulations of physical samples that were the focus of prior experimental work \cite{gregg2018tomographic}.  All relevant details can be found in the reference, however a brief description of each sample is as follows;
\begin{enumerate}
\item{\emph{Crushed Ring}: A sample formed by plastically deforming an initially stress-free steel ring along its diameter.  The geometry of the sample and applied deformation is shown in Figure \ref{Samples}a.  The residual strain field in this sample originates from a distributed eigen-strain related to plastic deformation (see Figure \ref{CRFEA}a)}
\item{\emph{Offset Ring-and-Plug}: A cylindrical steel sample constructed by shrink-fitting an oversize cylindrical `plug' into an undersize hole that is offset from the centreline (see Figure \ref{Samples}b). The strain field within this sample originates from the interference between the offset ring and the plug (see Figure \ref{RPFEA}a).  In the context of (\ref{strain_decomp}), the interference imposes a discrete eigen-strain with localised support on the interface.}
\end{enumerate}

\begin{figure} [hbt]
\begin{center}
    \includegraphics[width=0.6\linewidth]{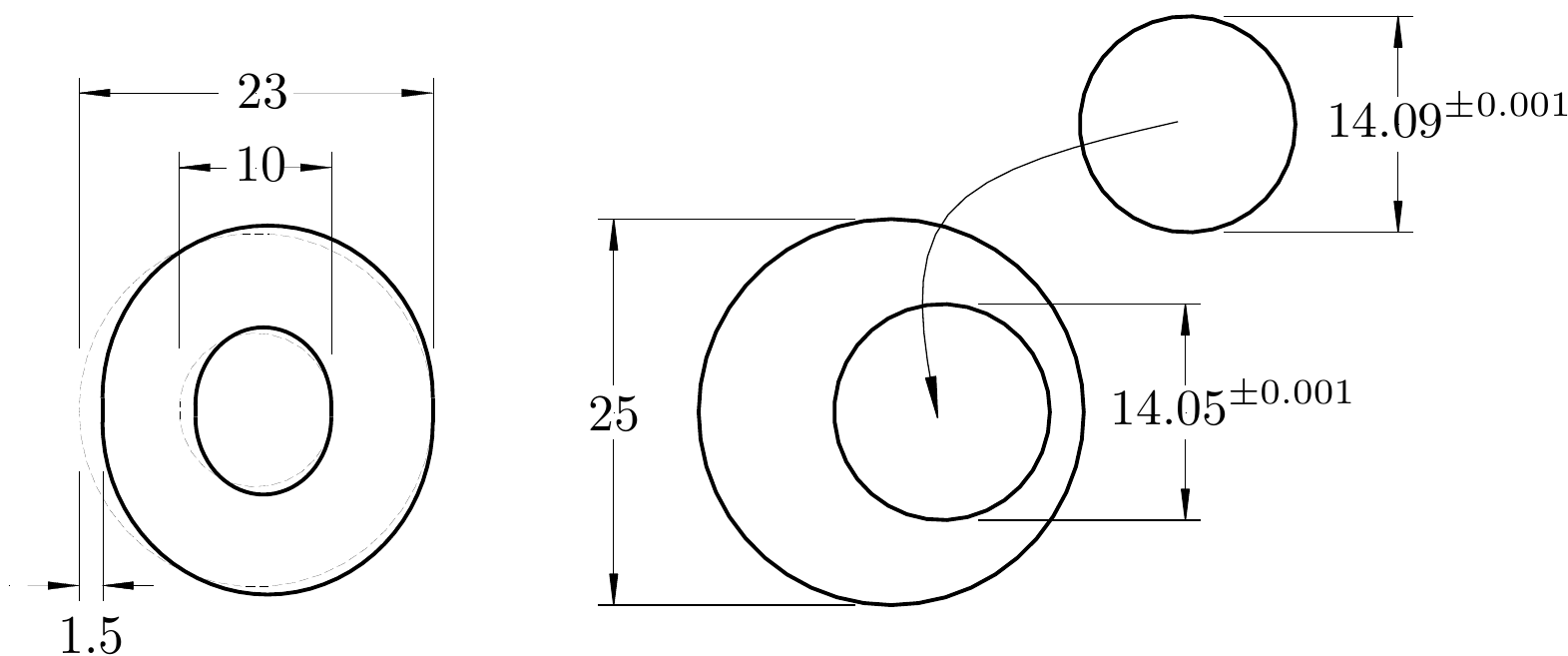}
\put(-260,80){(a) \hspace{17ex} (b)}\\
    \caption{Two samples representing strain fields used to perform numerical demonstrations of the reconstruction algorithm. (a) A crushed steel ring containing a distributed eigen-strain field. (b) An offset ring and plug system containing a discrete eigen-strain field generated through mechanical interference.}
\label{Samples}
\end{center}
\end{figure}

\begin{figure}[htb]
\begin{center}
    \includegraphics[width=0.24\linewidth]{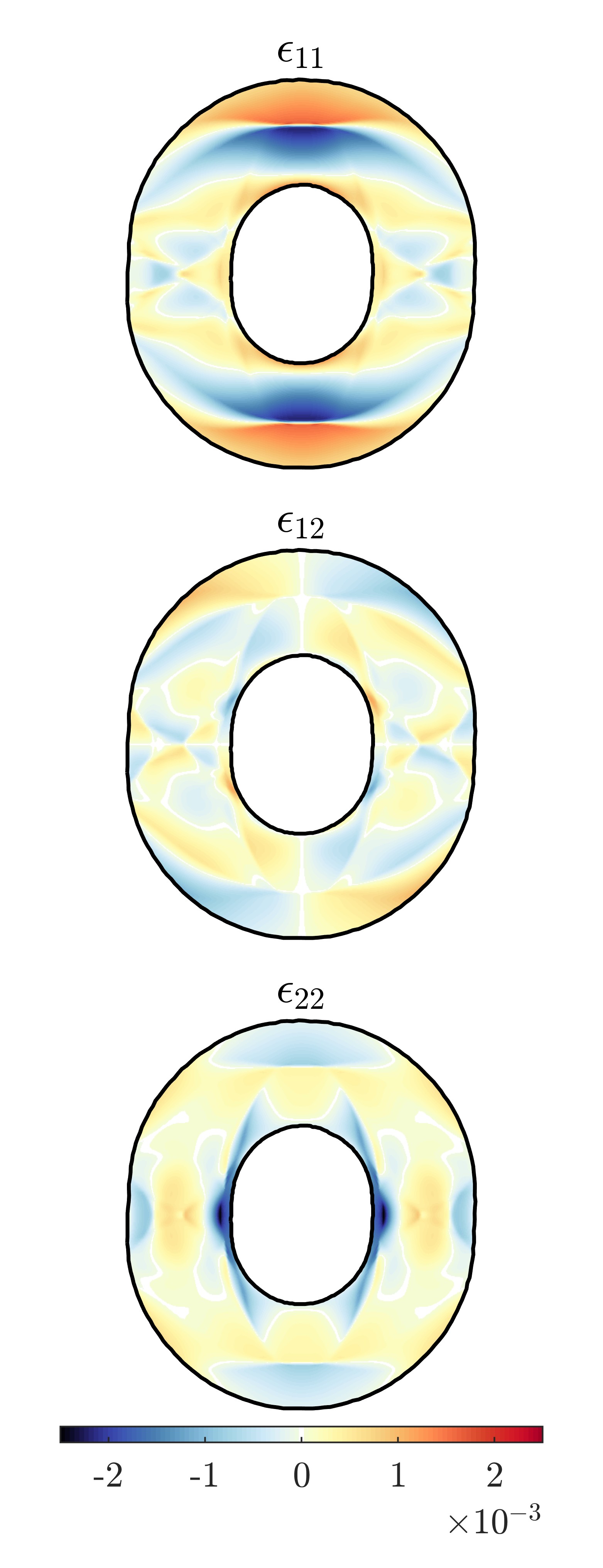}
    \includegraphics[width=0.24\linewidth]{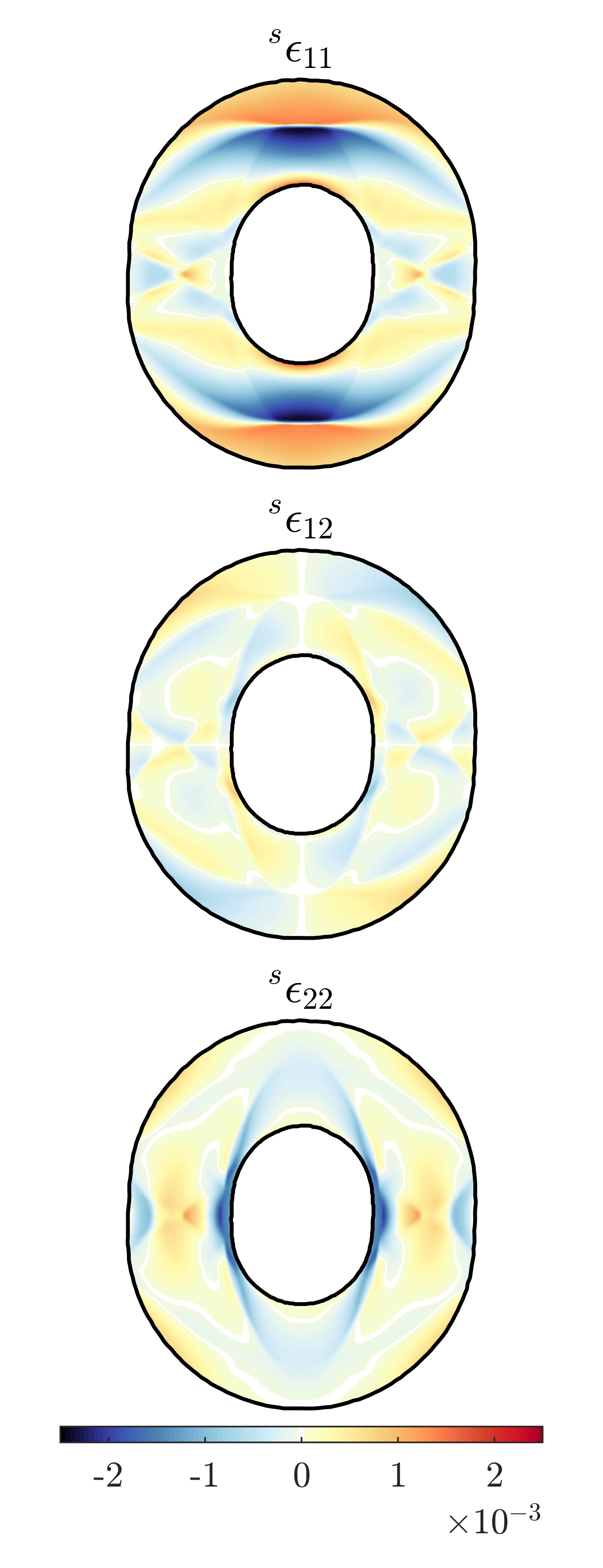}
    \includegraphics[width=0.24\linewidth]{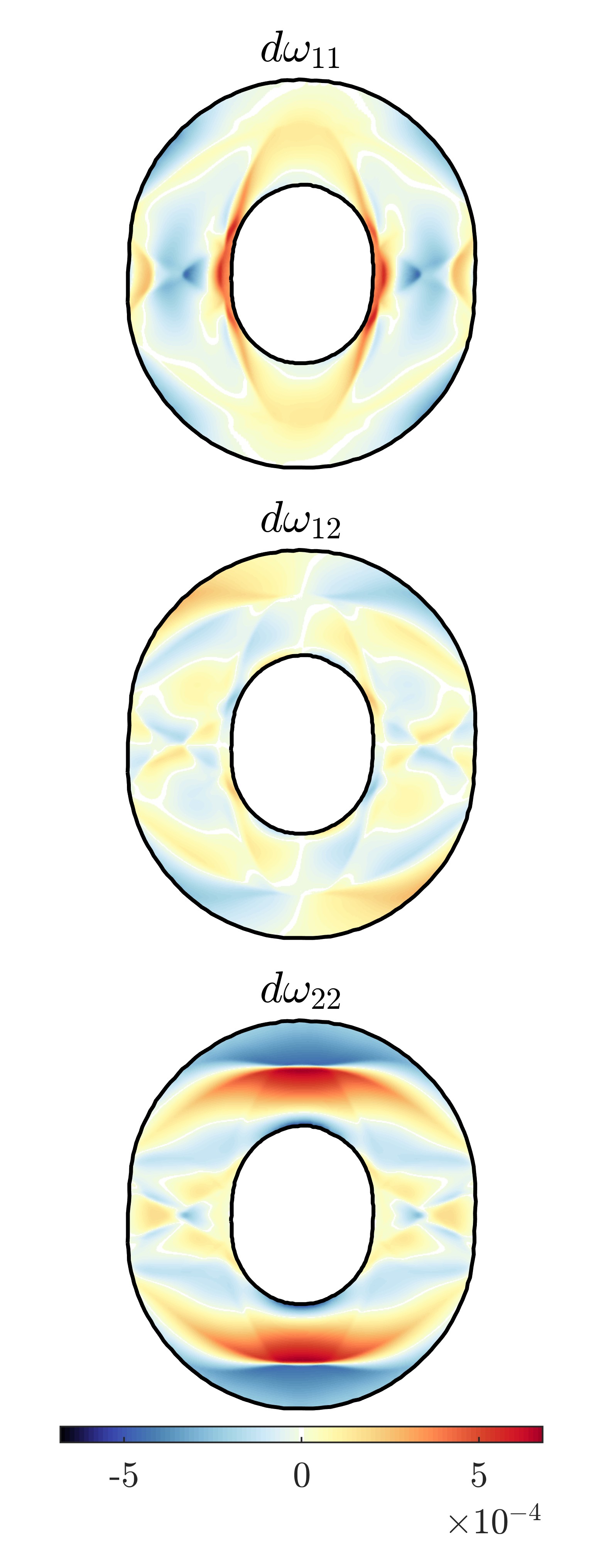}
    \includegraphics[width=0.24\linewidth]{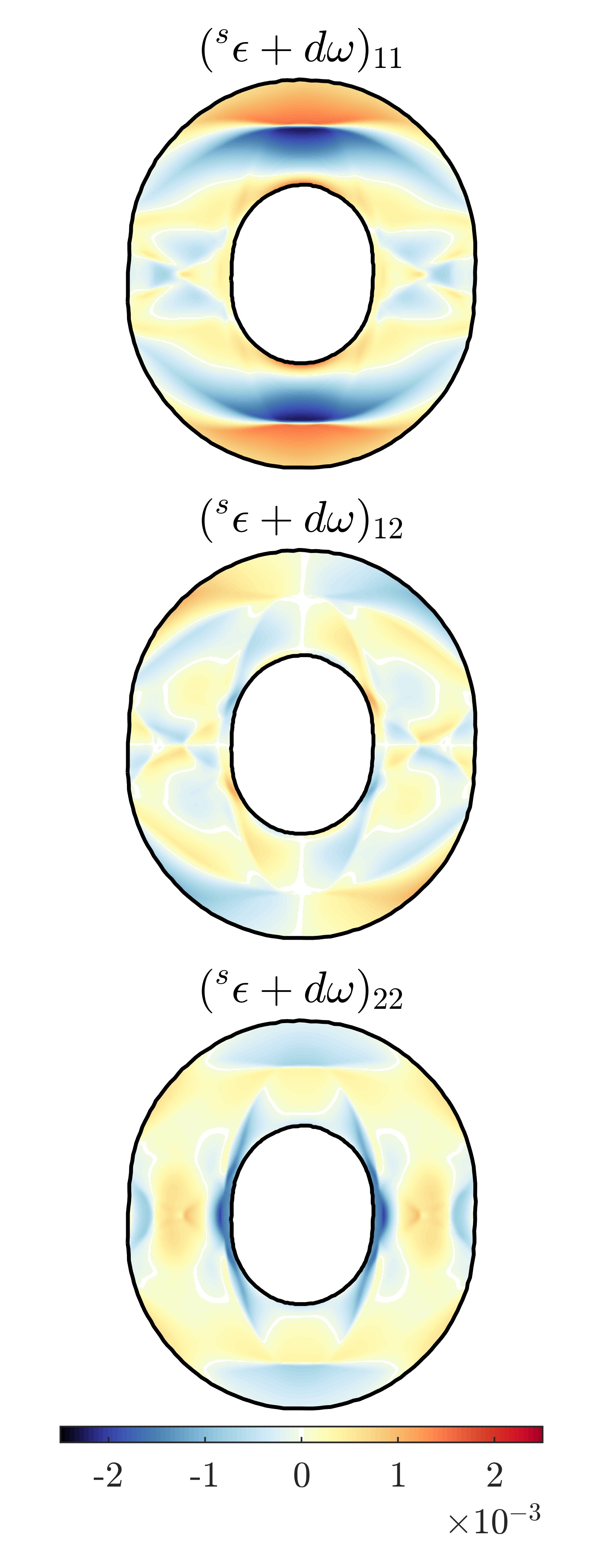}
    \put(-380,230){(a)}
    \put(-282.5,230){(b)}
    \put(-183.5,230){(c)}
    \put(-85,230){(d)}
    \caption{A reconstruction of a synthetic strain field computed from an elasto-plastic finite element model of the crushed ring. (a) The original strain field. (b) A reconstructed of the solenoidal component of this field from a simulated LRT consisting of 200 equally spaced projections over 360$^\circ$. (c) The recovered potential component from elastic finite element modelling. (d) The reconstructed strain field formed by the sum of the solenoidal and potential components.}
    \label{CRFEA}
\end{center}
\end{figure} 

\begin{figure}[htb]
\begin{center}
    \includegraphics[width=0.24\linewidth]{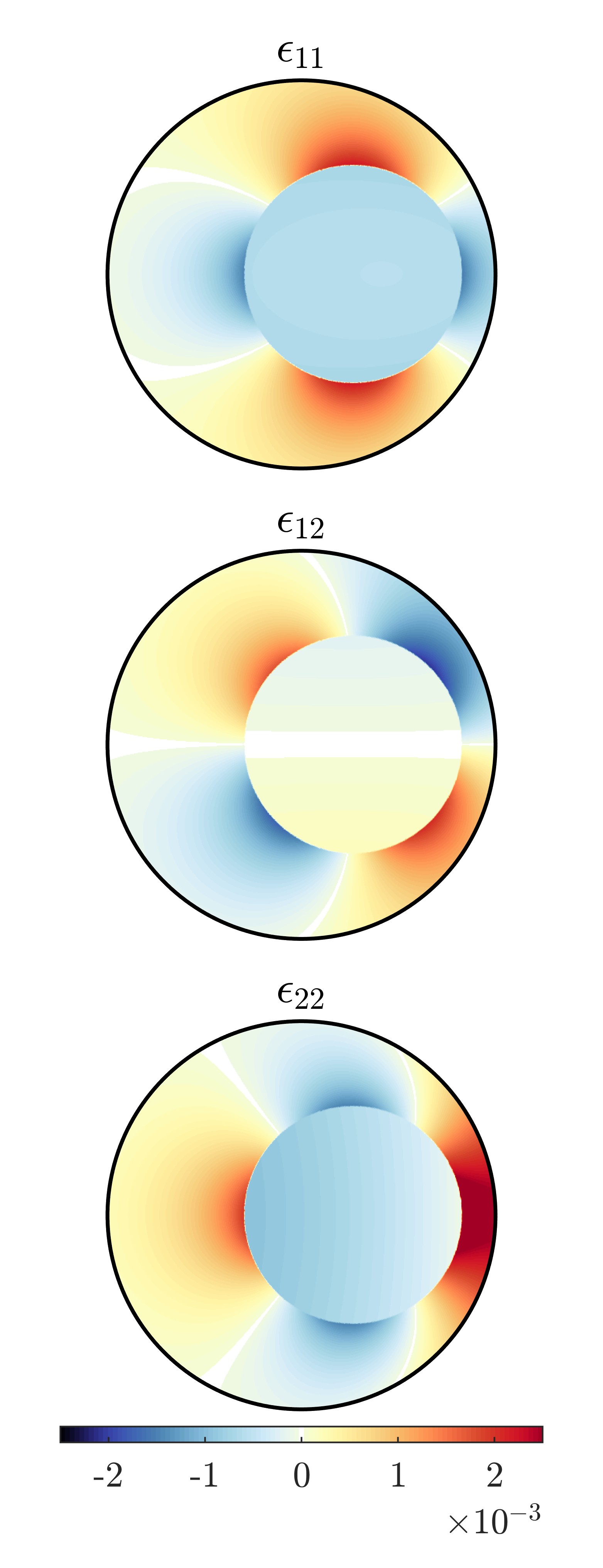}
    \includegraphics[width=0.24\linewidth]{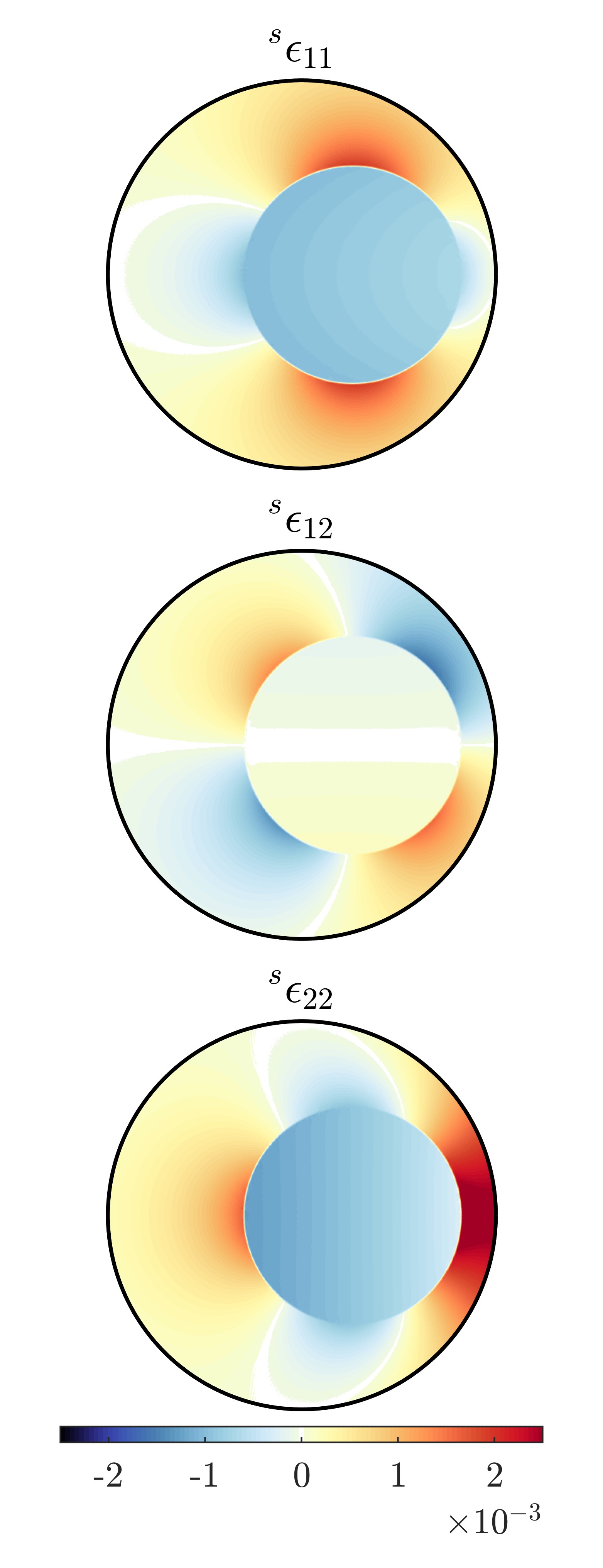}
    \includegraphics[width=0.24\linewidth]{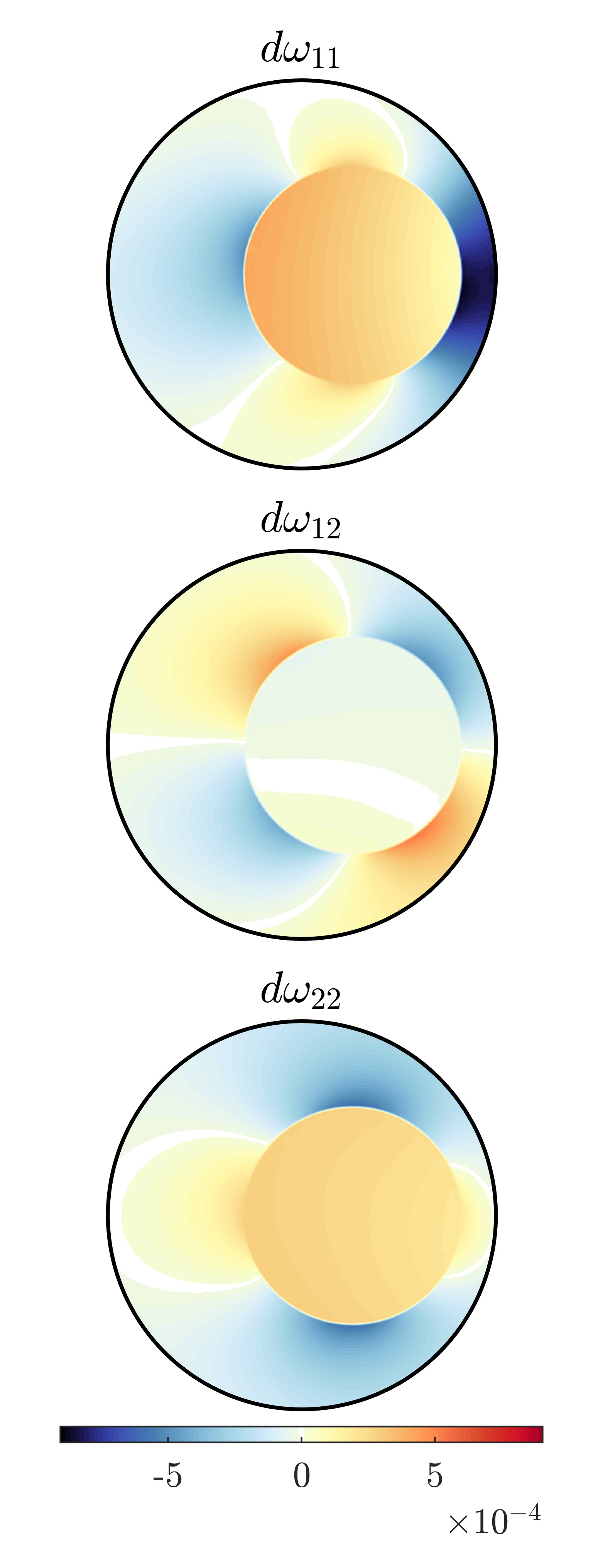}
    \includegraphics[width=0.24\linewidth]{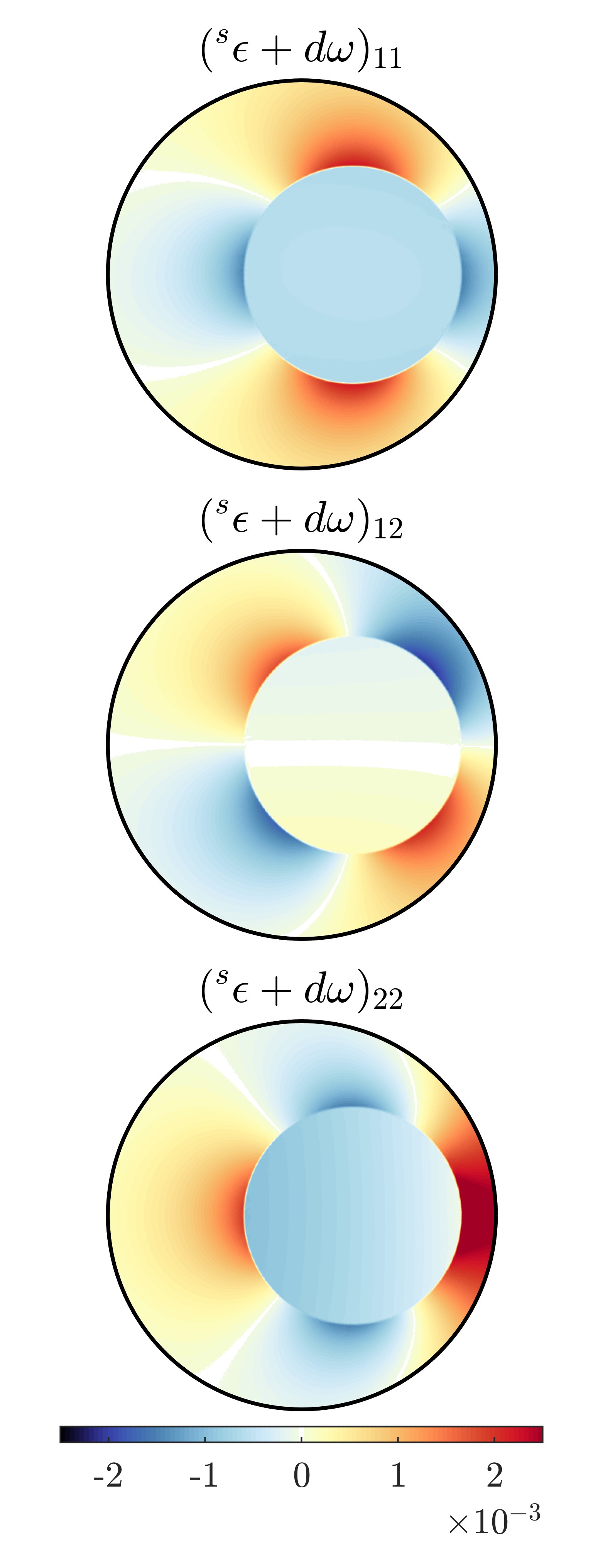}
    \put(-380,230){(a)}
    \put(-282.5,230){(b)}
    \put(-183.5,230){(c)}
    \put(-85,230){(d)}
    \caption{A reconstruction of a synthetic strain field computed from an linear-elastic finite element model of the offset ring and plug system. (a) The original strain field. (b) A reconstructed solenoidal component of this field from a simulated LRT consisting of 200 equally spaced projections over 360$^\circ$. (c) The recovered potential component from elastic finite element modelling. (d) The reconstructed strain field formed by the sum of the solenoidal and potential components.}
    \label{RPFEA}
\end{center}
\end{figure} 

Both samples were 14mm thick and were simulated as steel with $E=209$GPa, $\nu=0.34$ and a yield stress of 650MPa.  The finite element model for the first sample required a non-linear solve based on an elasto-plastic material model, while the second sample was modelled using linear-elasticity.  Both models were built and solved in the software package PTC/Creo.

All three strain fields were represented as three scalar components mapped to regular two-dimensional grids.  The size and resolutions of these grids were as follows: Airy -- $400\times400$, spacing 0.006, Crushed Ring -- $500\times500$, spacing 48$\mu$m, Ring and Plug -- $521\times521$, spacing 50$\mu$m.  In each case, all three strain components were extended by zero outside the sample boundaries.

What follows is a demonstration of the reconstruction of these fields from synthetic LRT data.  

\subsection{Procedure}

The demonstrations were was carried out with the help of the Matlab `\texttt{radon}' and `\texttt{iradon}' functions.  In this context, the implementation was as defined in the following process:

\begin{enumerate}
\item{Forward map the LRT of the strain field by successive application of the `\texttt{radon}' Matlab function for each individual projection angle.  i.e. for a given projection angle $\theta$:
\[
I\epsilon(s,\theta)=\mathcal{R}[\cos^2\theta\epsilon_{11} +2 \cos\theta\sin\theta\epsilon_{12}+\sin^2\theta\epsilon_{22}]
\]}
\item{Component-wise back-project the resulting strain-sinogram to compute the three unique components of $^s\epsilon$ using the FBP algorithm as implemented in the `\texttt{iradon}' intrinsic Matlab function (as per \eqref{InvLRT}).}
\item{Calculate a first reconstruction of $\epsilon$ from $^s\epsilon$ based on Hooke's law using \eqref{HookeRecon1}, \eqref{HookeRecon2} and \eqref{HookeRecon3}.}
\item{Calculate derivatives of $^s\epsilon$ by first transforming the individual components to the Fourier domain using the `\texttt{fft2}' and `\texttt{fftshift}' intrinsic Matlab functions.  These transformed components are then multiplied by appropriate $\kappa$-space filters corresponding to $\partial/\partial x_1$ and $\partial/\partial x_2$ before transforming back to the real domain using `\texttt{fftshift}' and `\texttt{ifft2}'}
\item{From these derivatives, calculate the two components of the vector $b$ using \eqref{bx} and \eqref{by}.}
\item\label{FEAmodel}{Using the Matlab PDE solver, calculate a finite element solution for the displacement field $\omega$ satisfying (\ref{OmegaEqu}) and (\ref{OmegaBC}) subject to the calculated vector field $b$.}
\item{Calculate a second reconstruction for $\epsilon$ as the sum $\epsilon={^s\epsilon}+d\omega$, where $d\omega$ is computed from the shape functions within the finite element solution.}
\end{enumerate}

The target element size for the finite element model in step \ref{FEAmodel} was set to be 0.5\% of the maximum sample dimensions.  This was conservatively chosen through a standard mesh-independence investigation.

\subsection{Results}

In all three cases the reconstructions based on Hooke's law and the finite element recovery of the potential component were visually indistinguishable from each other. However, the reconstruction based on Hooke's law was slightly more accurate in terms of a root-mean-square error.  

Figures \ref{AiryField}, \ref{CRFEA} and \ref{RPFEA} show the results of this process based on simulated LRT data from 200 equally spaced angular projections over 360$^\circ$.  Each figure shows the original strain field together with the reconstructed solenoidal component, the recovered potential component, and the final reconstruction based on the sum of the two.

 It was also interesting to note that, in each case, the reconstructed solenoidal component was approximately zero outside the sample boundary (as expected from Lemma \ref{SupportLemma}).  This is examined further in Section \ref{Support} below.

\begin{figure}[h]
\begin{center}
    \includegraphics[width=0.32\linewidth]{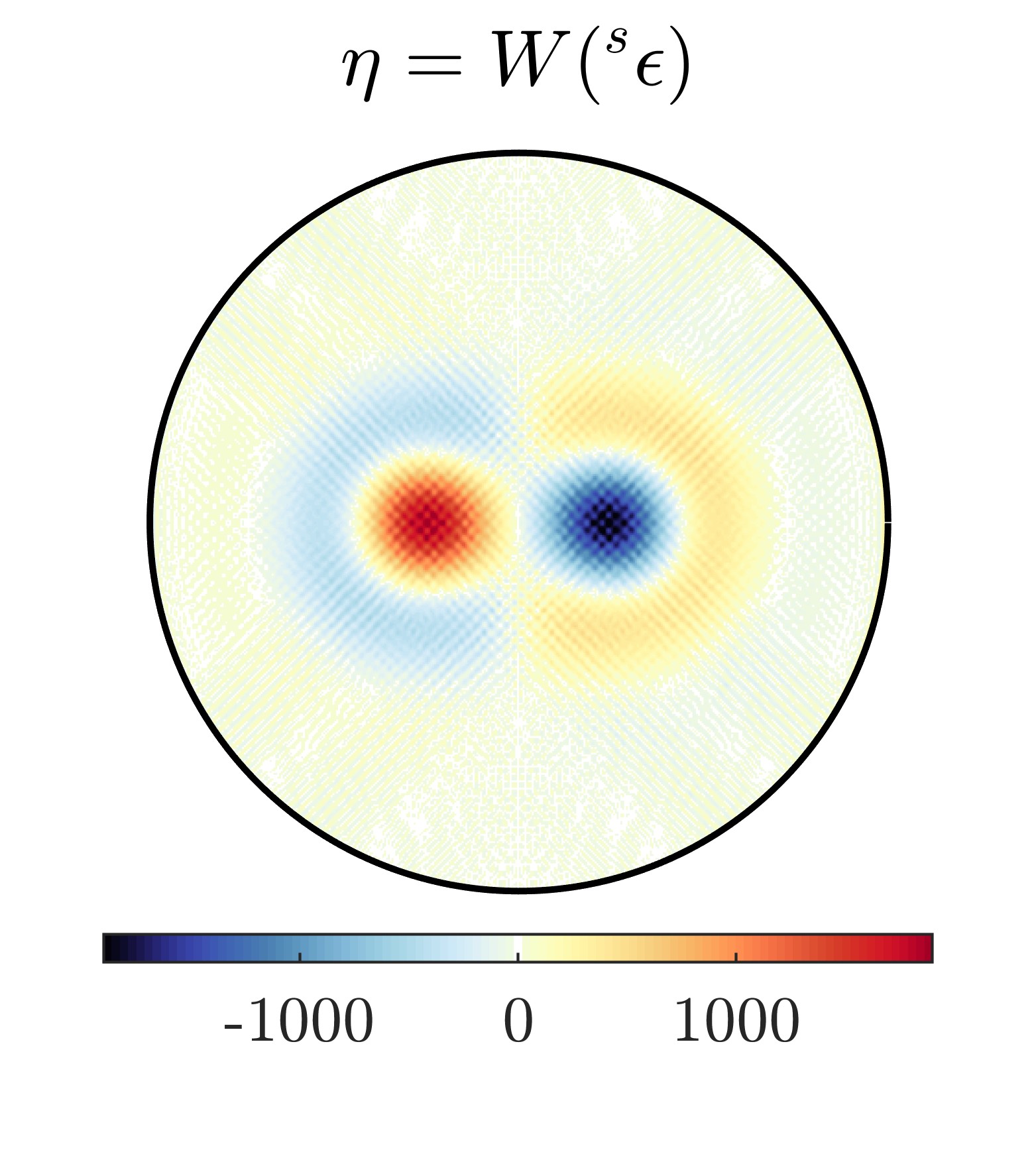}
    \includegraphics[width=0.32\linewidth]{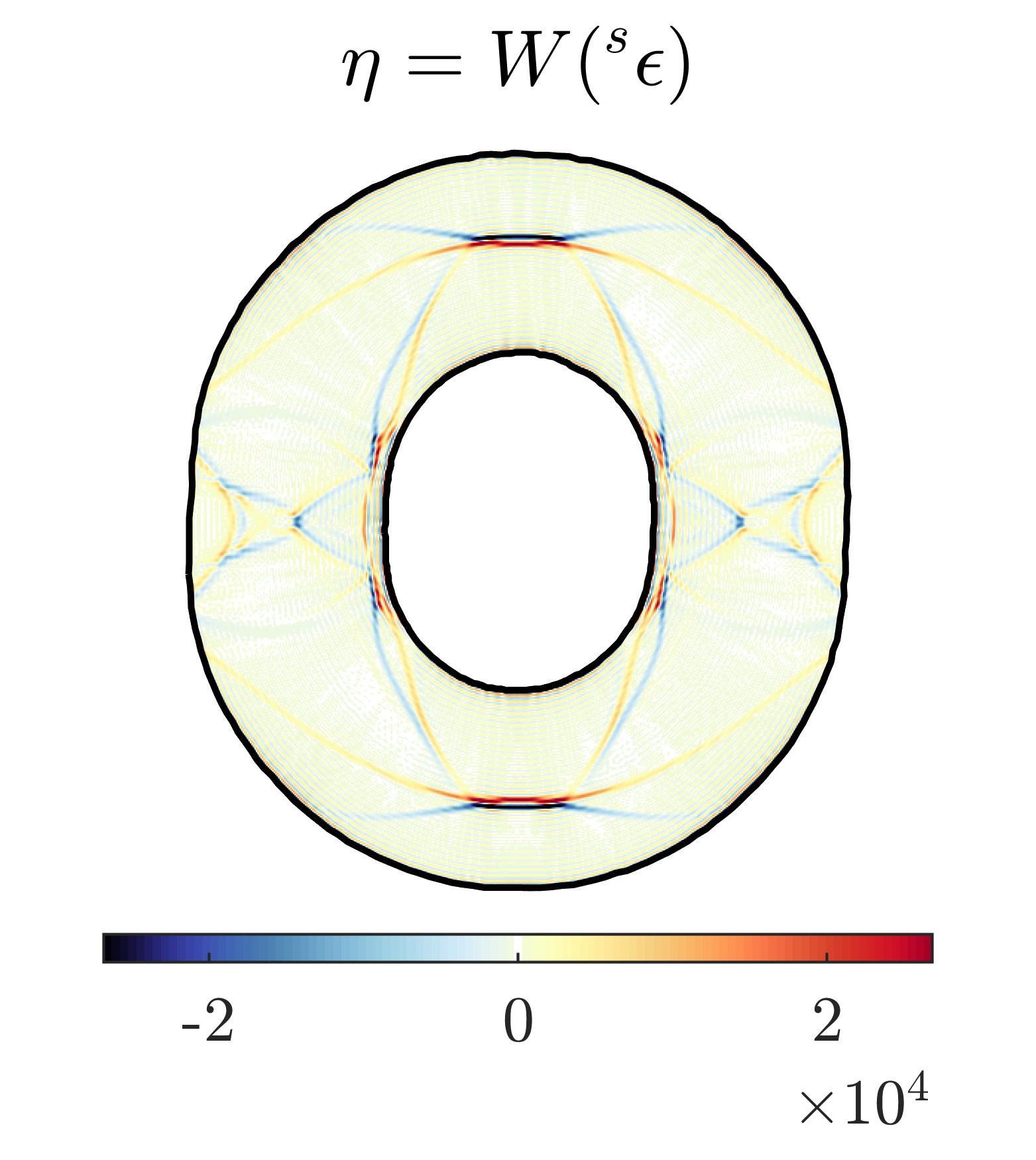}
    \includegraphics[width=0.32\linewidth]{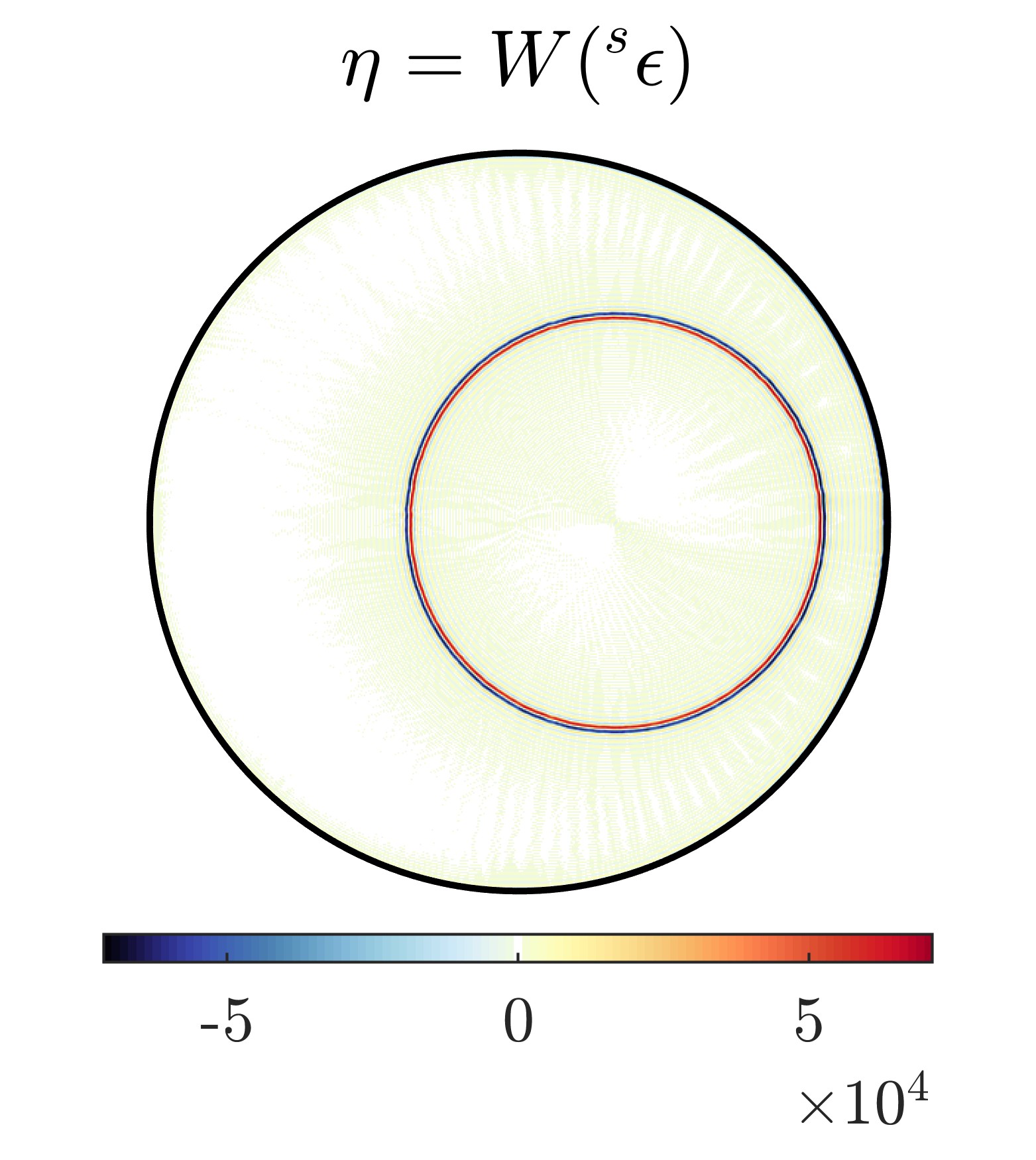}\\
    \includegraphics[width=0.32\linewidth]{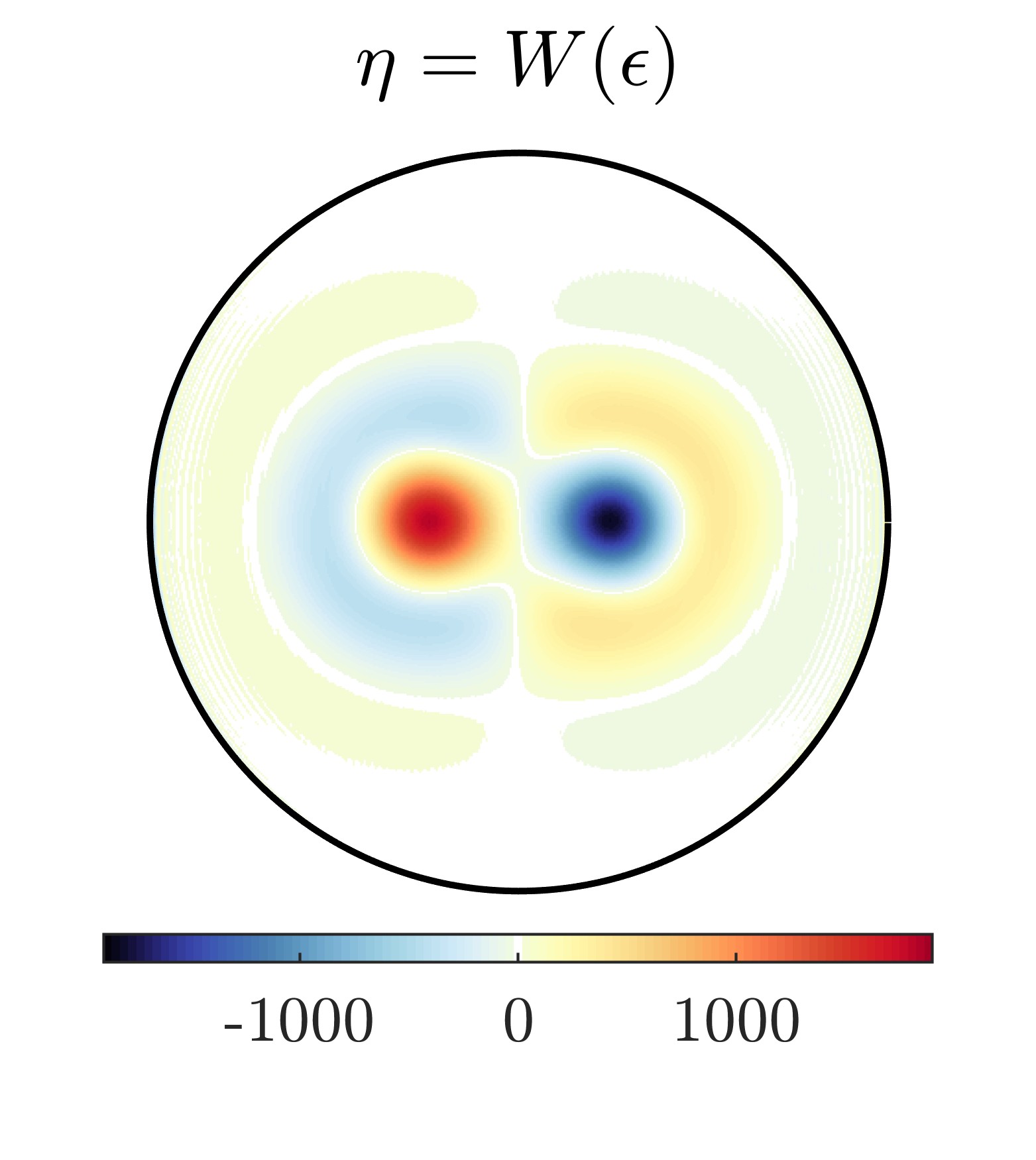}
    \includegraphics[width=0.32\linewidth]{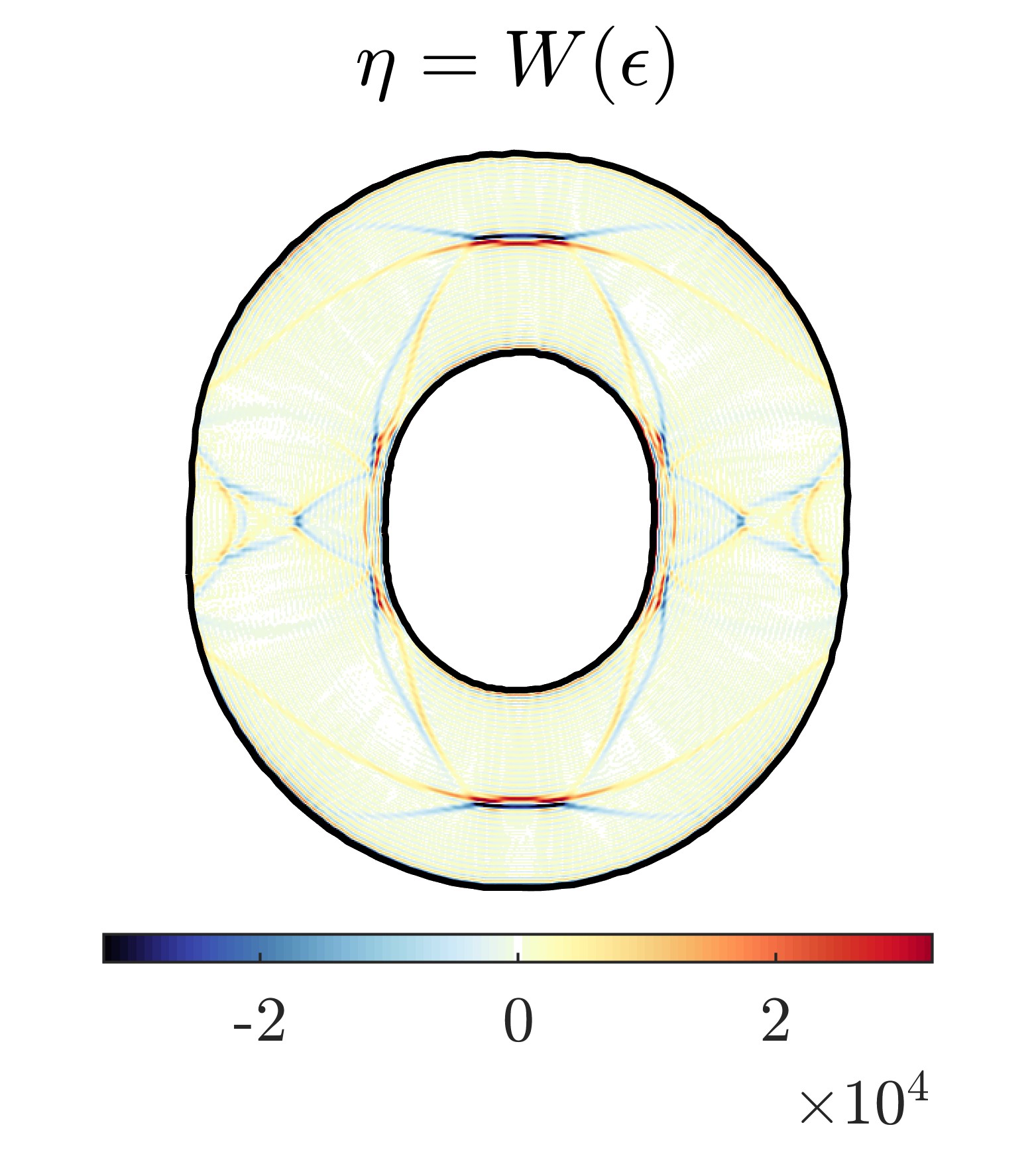}
    \includegraphics[width=0.32\linewidth]{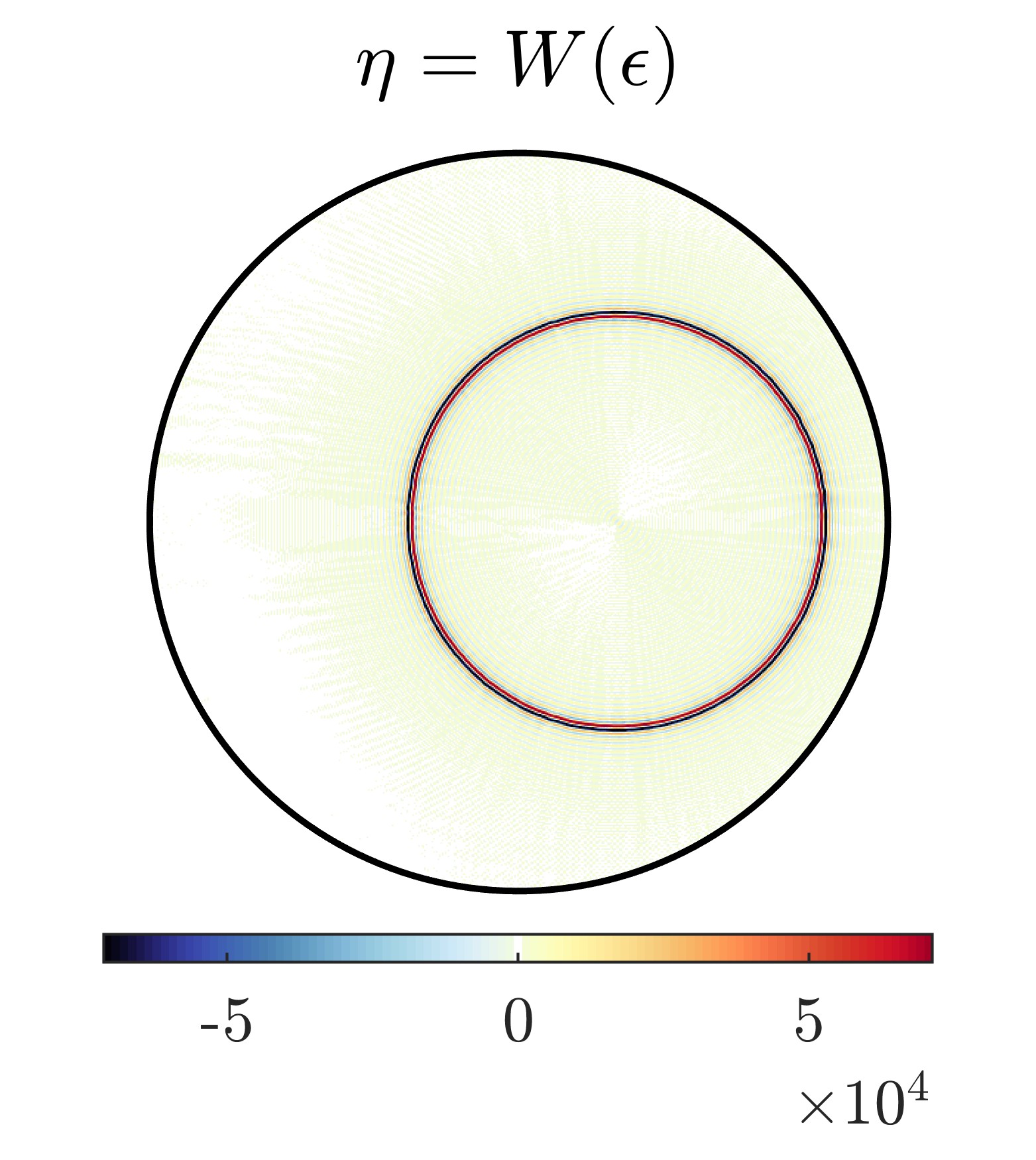}
    \caption{The Saint-Venant operator as applied to the reconstructed solenoidal components compared to the same for the original strain fields. }
    \label{Etaplot}
\end{center}
\end{figure} 

The difference between the reconstructions and the original field was small; typically around 1-5\% of the maximum value of the original components.  However, it was observed that this did not significantly decrease along with the number of projections.  The source of this persistent discrepancy was discretisation error related to minor deviations from the equilibrium relation introduced by various interpolations onto the regular grid.  This is examined further in the following section.

Figure \ref{Etaplot} shows the computed Saint-Venant incompatibility of the reconstructed solenoid compared to the original for all three fields.  These images were calculated using a similar transform-filter-transform approach in the Fourier domain. 

The Airy stress field shows incompatibility distributed over the sample domain, whereas the other two samples show more localised support.  In the case of the crushed-ring, this is likely to have originated from localised plastic shear within the elasto-plastic finite element model, while the offset ring-and-plug indicates a clear dipole around the circumference of the plug corresponding to the interference.

As expected, the incompatibility of the reconstructed solenoidal components are identical to that of the original fields within a small amount of numerical noise.

\subsection{Reconstruction in the presence of measurement uncertainty}

A further set of simulations was carried out in order to examine the behaviour of reconstructions in the presence of Gaussian noise.  In this respect both approaches were found to be quite stable and converged to the original field with an increasing number of projections (notwithstanding the discretisation error identified earlier).  

Although not strictly necessary, slight improvement was found by limiting the order of terms in the numerical derivatives used to compute $b$.  This was achieved by cutting-off the $\kappa$-space filters for frequencies above a certain threshold. A cut-off frequency equal to 0.7 times the maximum magnitude provided a good compromise between noise and fidelity.

\begin{figure}[h]
\begin{center}
    \includegraphics[width=0.95\linewidth]{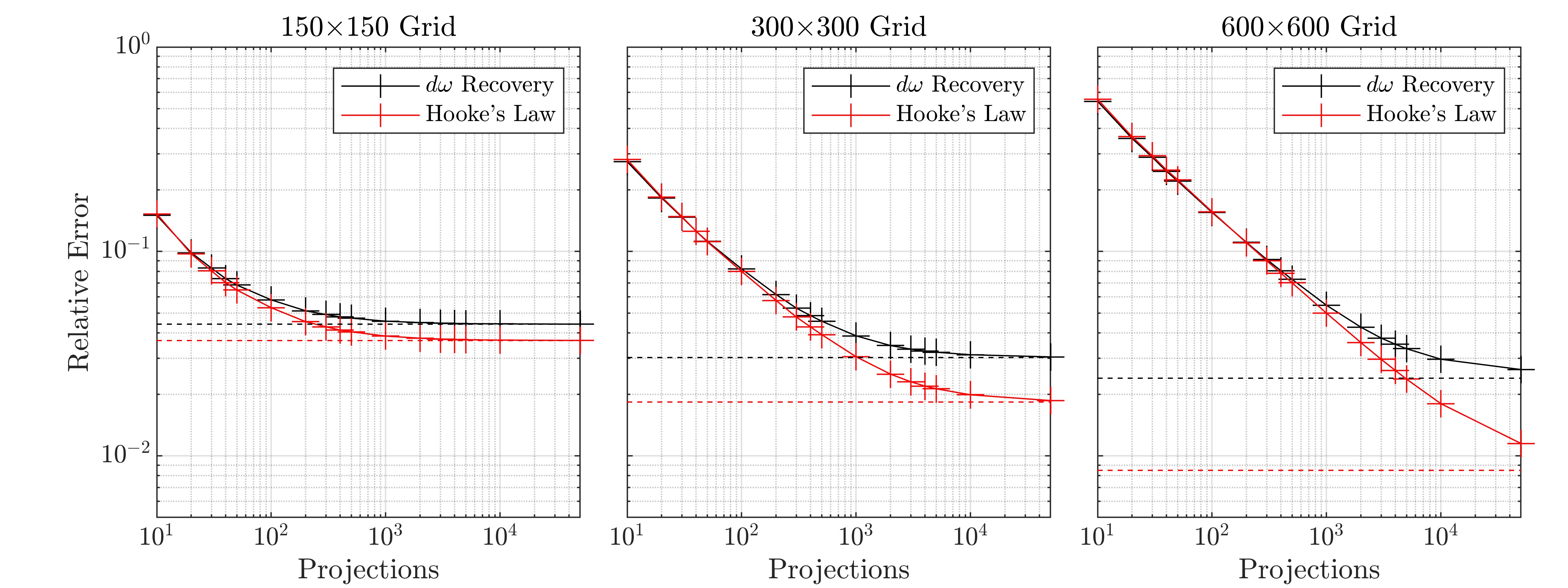}
    \caption{The overall error in the reconstruction of the Airy stress field in the presence of 10\% Gaussian measurement noise as a function of the number of projections.  The relative error is computed as the root-mean-square of the residual divided by the root-mean-square of the original strain field summed over all components. Dotted lines show the minimum error possible for the given mesh density (calculated using 50,000 projections with no added noise). }
    \label{RelErr}
\end{center}
\end{figure} 

For the Airy stress field, Figure \ref{RelErr} shows the convergence of the reconstructed fields along with the number of projections in the presence of Gaussian random noise with a standard deviation of 10\% of the maximum LRT value.  Results from three systems are shown corresponding to different spatial resolutions (i.e. grid size).  In each case, convergence of the relative error to zero is observed to occur at $\mathcal{O}(n^{-1/2})$ until the lower limit corresponding to the discretisation error is reached.  

Generally speaking, the reconstruction based on Hooke's law had a lower persistent error and the size of the persistent error was observed to be directly related to the resolution of the grid.

It should be noted that, in the presence of noise the calculation of the Saint-Venant operator was found to be inherently unstable regardless of any reasonable cut-off frequency used in the relevant filters. 

\subsection{Boundary traction and compact support} \label{Support}

In order to examine the effect of the boundary conditions, a further set of simulations were carried out on the strain field specified in Appendix A of Gregg \textit{et al} \cite{gregg2017tomographic} with $e_0=R=1$ (see Figure \ref{BCFigure}a).  This is an axi-symmetric `plane-stress' strain field on the unit disk originating from the hydrostatic eigen-strain
\[
\epsilon^*_{rr}=\epsilon^*_{\theta\theta}=(1-r)^2,
\]
and subject to a zero traction boundary condition (i.e. $\sigma_{rr}(1)=0$).  In polar coordinates it has the form
\begin{align}
    \epsilon_{rr}&=\frac{7+5\nu+(1+\nu)(9r-16)r}{12}-(1-r)^2 \\
    \epsilon_{\theta\theta}&=\frac{7+5\nu+(1+\nu)(3r-8)r}{12}-(1-r)^2.
\end{align}

\begin{figure}[h!]
\begin{center}
    \includegraphics[width=0.24\linewidth]{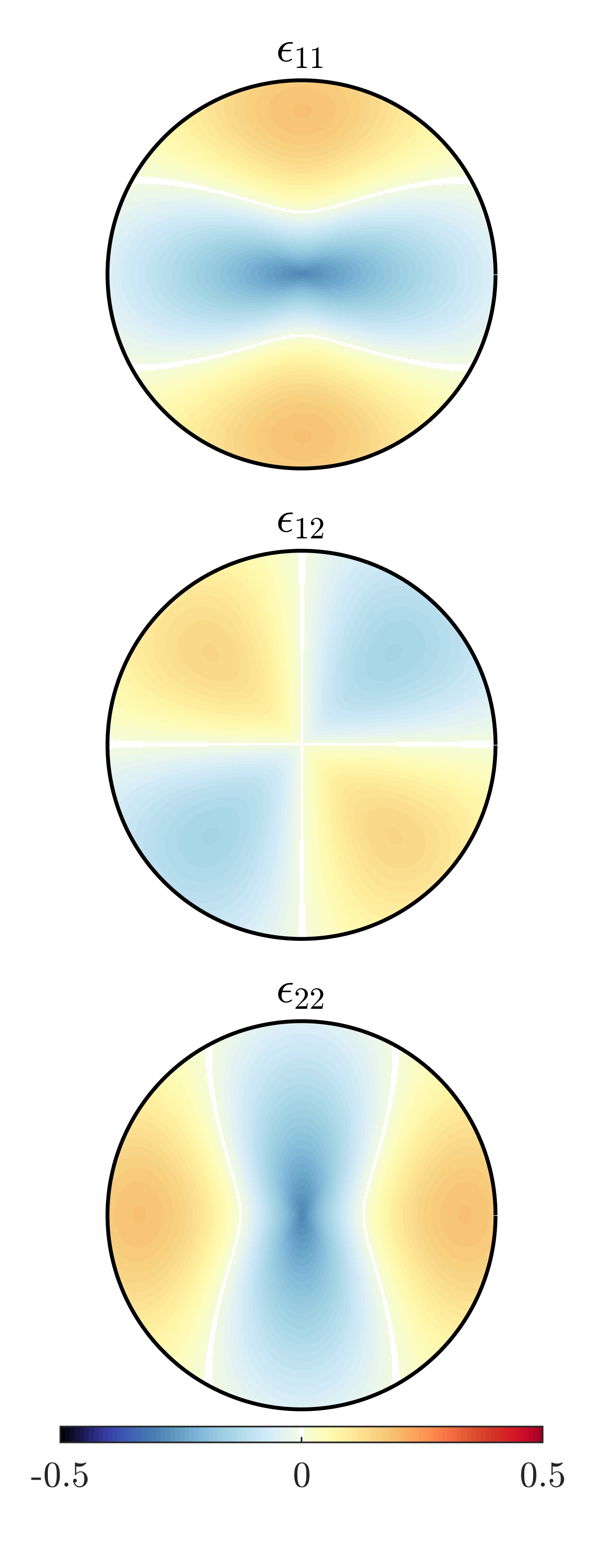}
    \includegraphics[width=0.24\linewidth]{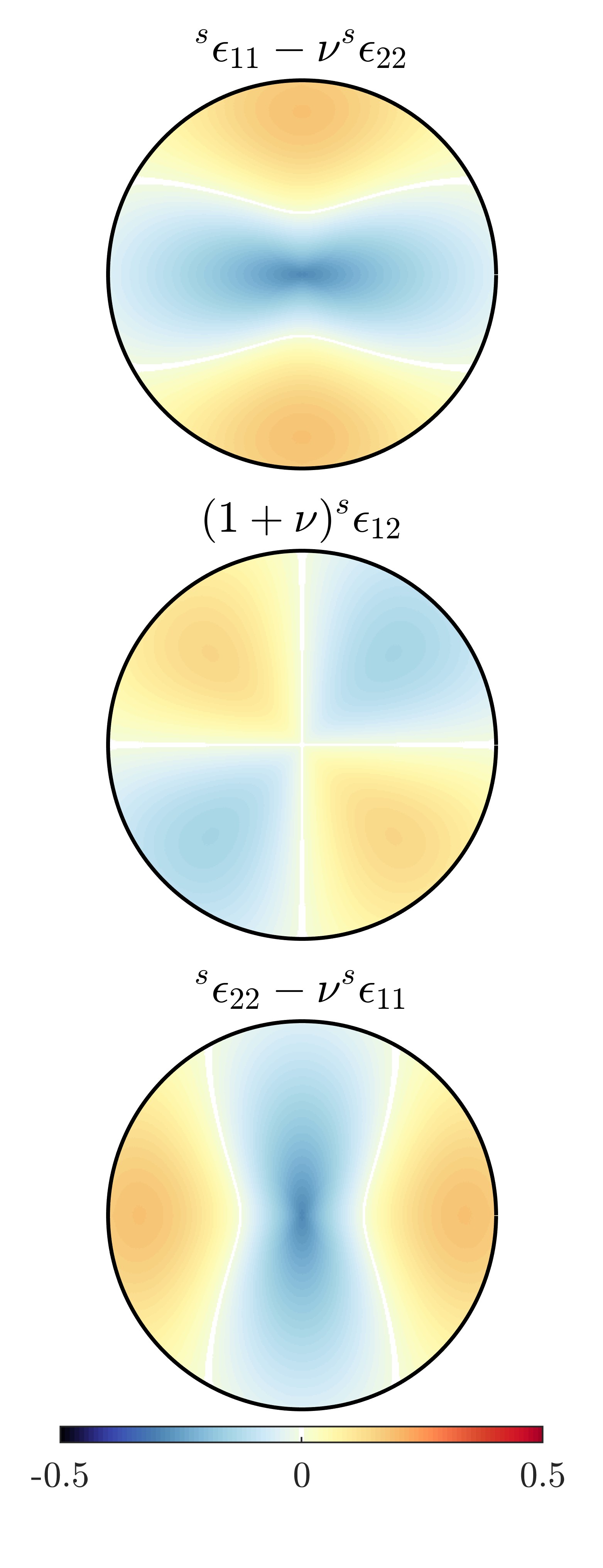}
    \begin{tikzpicture}
        \draw (-10,0) -- (-10,8.5);
    \end{tikzpicture}
    \includegraphics[width=0.24\linewidth]{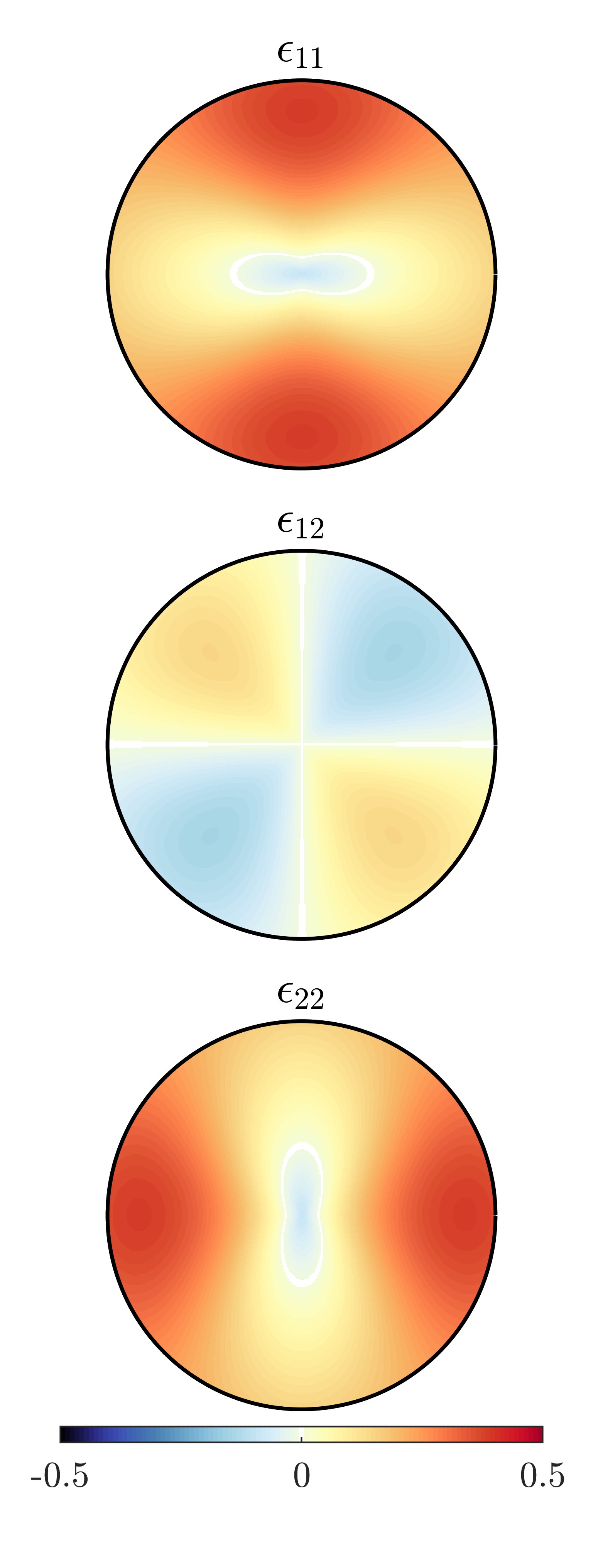}
    \includegraphics[width=0.24\linewidth]{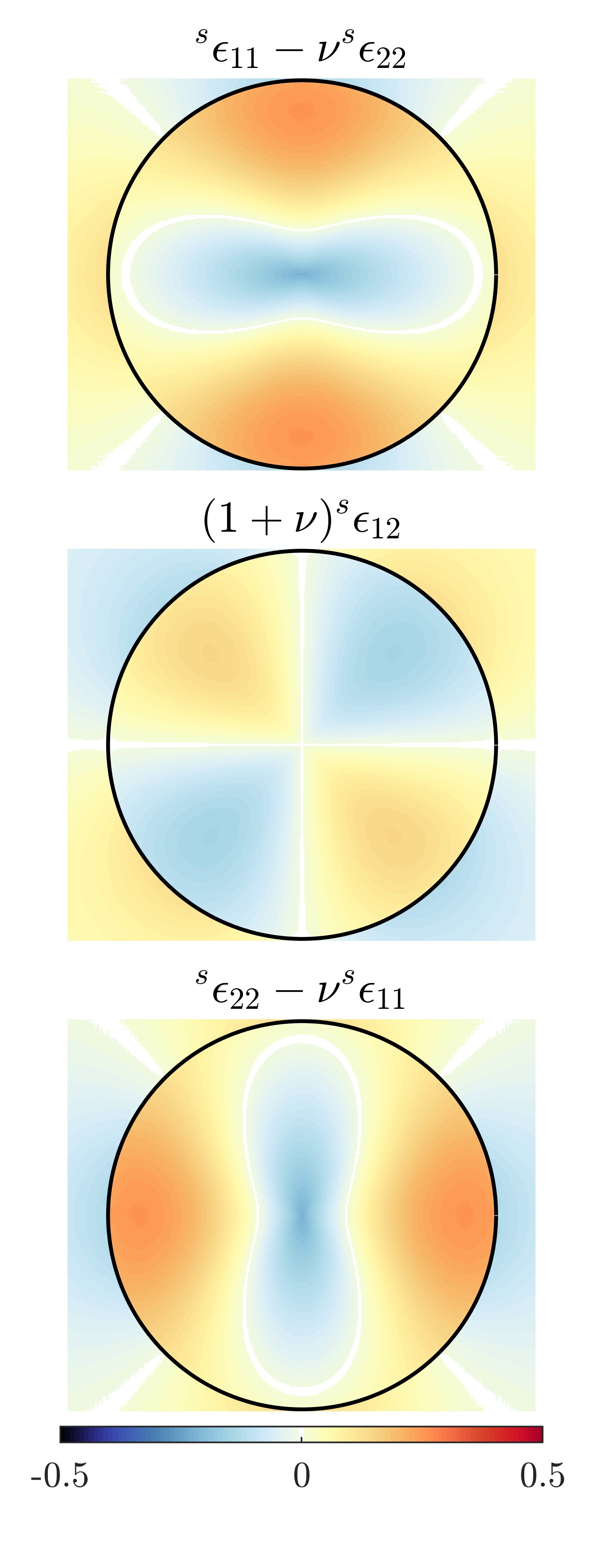}
    \put(-380,225){(a)} \put(-282.5,225){(b)} \put(-183.5,225){(d)} \put(-85,225){(e)} \\
    \includegraphics[width=0.49\linewidth]{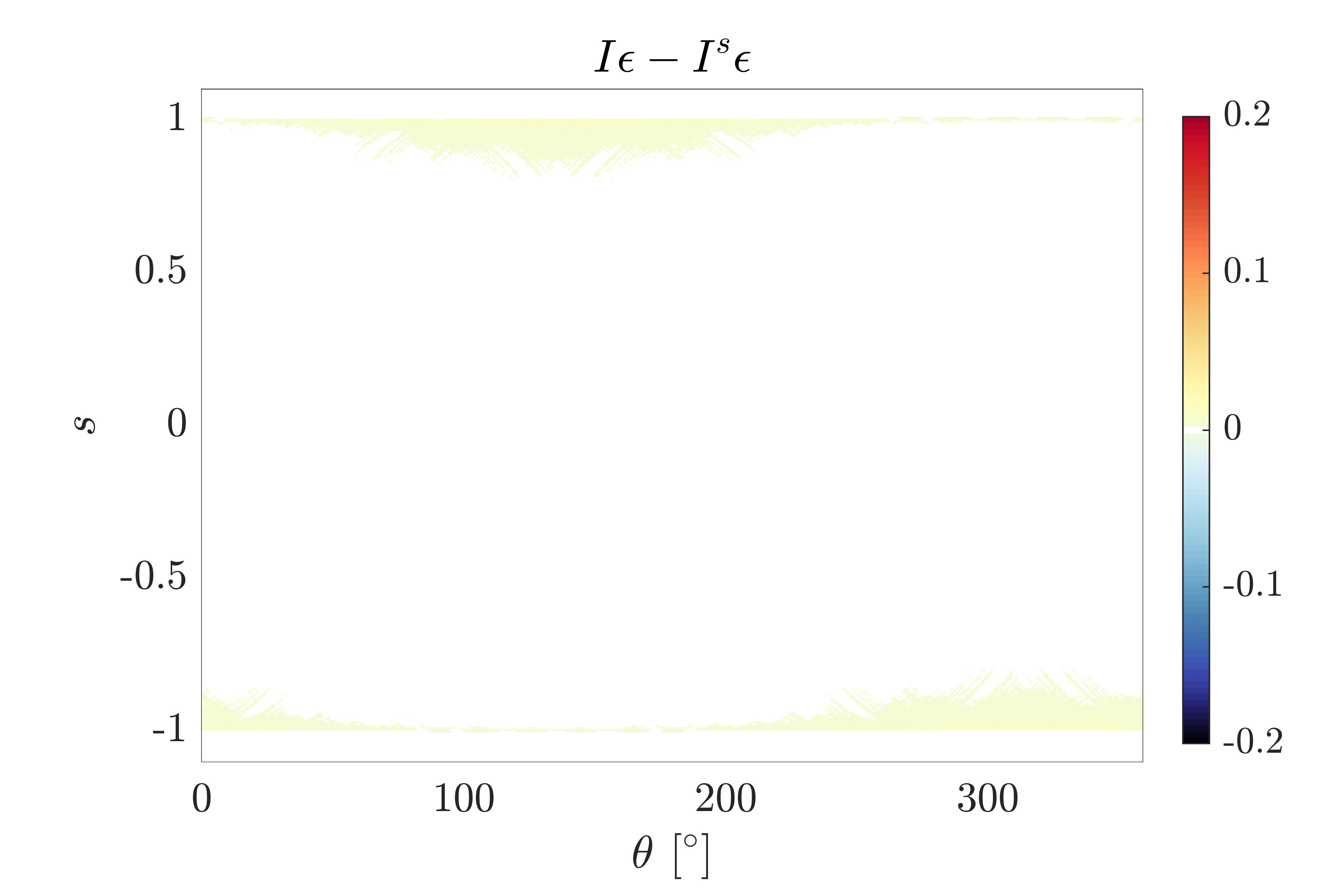}
    \begin{tikzpicture}
        \draw (-10,0) -- (-10,4.5);
    \end{tikzpicture}
    \includegraphics[width=0.49\linewidth]{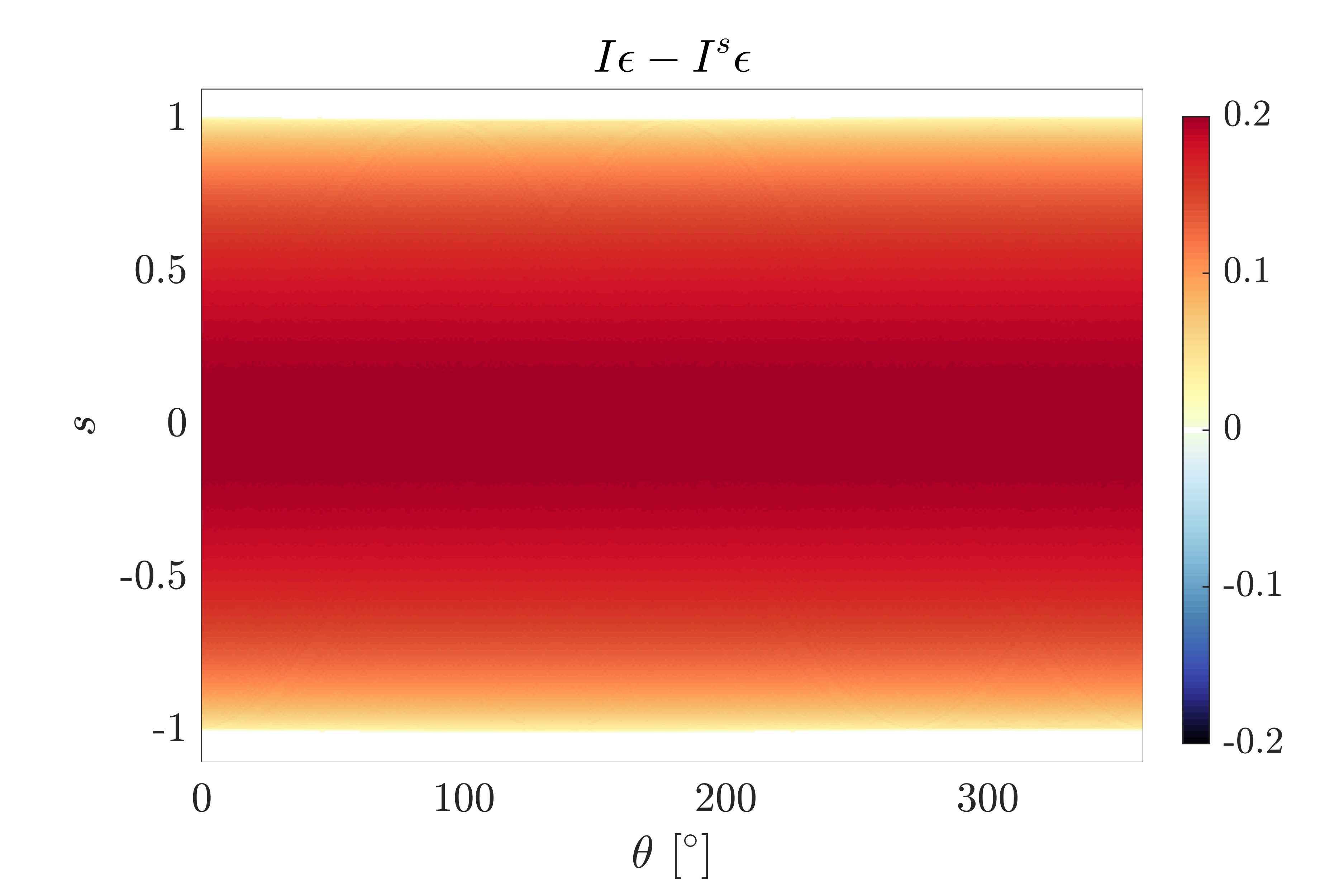} 
    \put(-380,115){(c)} \put(-185,115){(f)}
    \caption{Demonstration of the effect of the no-traction boundary condition on reconstruction. (a) An axisymmetric residual 'plane-stress' strain field that satisfies the no-traction boundary condition (see \cite{gregg2017tomographic}), along with (b) its reconstruction, and (c) the residual LRT between these fields. (d) The same field with an additional hydrostatic component that violates the no-traction condition, along with (e) a failed reconstruction and (f) the non-zero residual.}
    \label{BCFigure}
\end{center}
\end{figure}
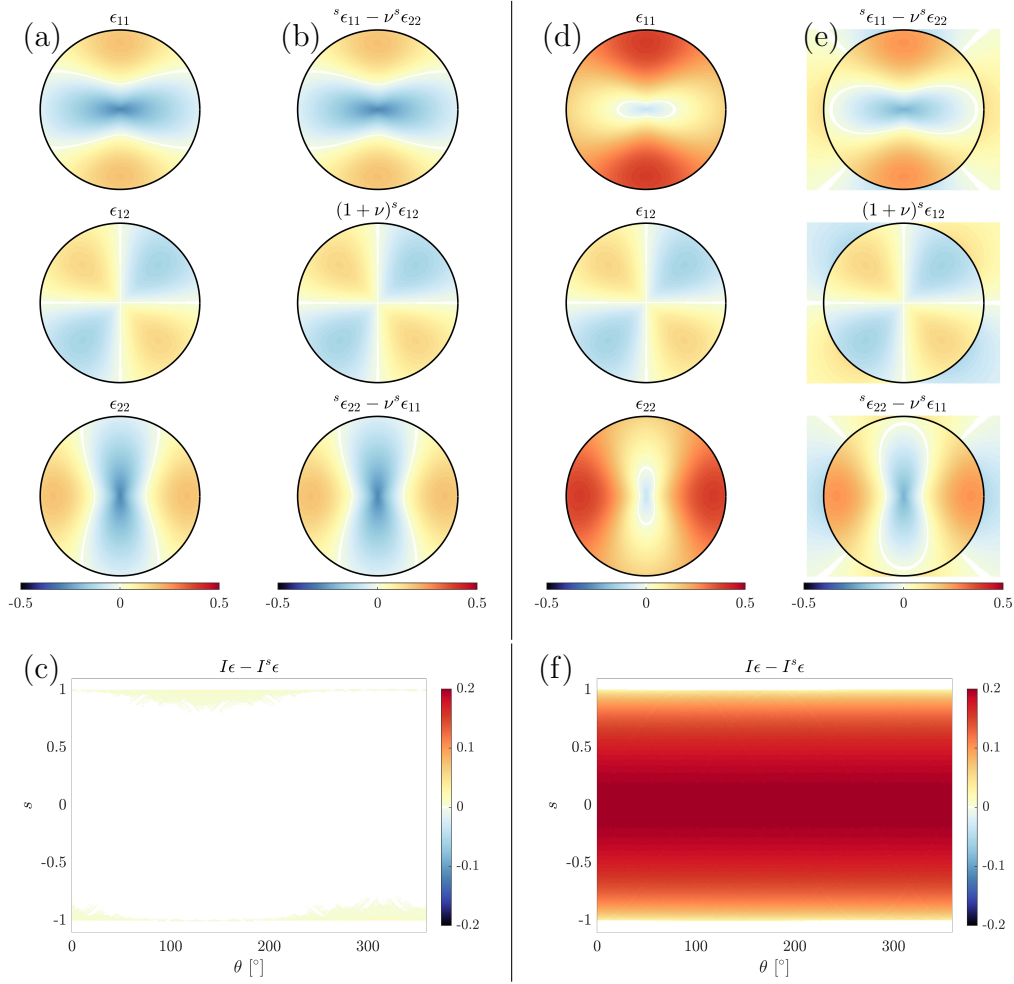 

A simulated reconstruction based on 1000 equally spaced LRT projections from a $400\times400$ Cartesian grid is shown in Figure \ref{BCFigure}b.  As expected, the reconstructed strain matches the original field accurately and the support of the reconstruction is contained within the boundary of the sample. Outside of the boundary, the reconstructed solenoidal component was around three orders-of-magnitude smaller than the original field.

Figure \ref{BCFigure}c shows the residual between the LRT of the original field and the reconstruction.

Figure \ref{BCFigure}d shows the same field with the addition of a constant hydrostatic strain of magnitude $\bar\epsilon=0.2$. Like the original field, this altered version satisfies equilibrium at all points within the sample, however it clearly violates the traction-free boundary condition since $| \sigma \cdot n |=0.2$ on $\partial\Omega$.

An attempted reconstruction of this field based on the same process is shown in Figure \ref{BCFigure}e.  A visual inspection of the result clearly indicates the reconstruction has failed to reproduce the original field.  

It is also interesting to note that the reconstructed field is far from zero outside the boundary of the sample.  This observation, together with Lemma \ref{lem_ext} suggests that the apparent support of $^s\epsilon$ reconstructed from data gives a reliable indicator of the existence of a harmonic potential component, and hence the appropriateness of the traction-free assumption for a given experimental system.

It is also clear that the LRT of the reconstructed solenoid does not match that of the original field. Figure \ref{BCFigure}f shows the difference between these two sinograms computed with $^s\epsilon$ masked to zero outside the boundary.  The residual is of a significant magnitude and appears to correspond directly to the added hydrostatic/harmonic component. This poses an interesting question: Given the harmonic component is compatible, can it be recovered through reconstruction of a non-zero boundary condition similar to the process carried out by Hendriks \textit{et al} \cite{hendriks2017bragg}?  This question will form the focus of future work in this area.

\section{Numerical demonstration: Experimental data}

As a final demonstration, the reconstruction approach was applied to experimental data measured from physical samples using the RADEN energy resolved imaging instrument within the Materials and Life Sciences institute at the J-PARC spallation neutron source in Japan \cite{shinohara2020energy}.  All relevant details of this experiment are described in Gregg \textit{et al} \cite{gregg2018tomographic}.  The outcome of this experiment was measured strain-sinograms from the crushed-ring and offset ring-and-plug samples corresponding to a set of 50 golden-angle projections.  As per \eqref{BEStrain}, these measurements correspond to average strain along ray-paths, which require multiplication by appropriate values of $L$ to compute the LRT (see Figure \ref{RealData}a and \ref{RealData}b).

\begin{figure}[h!]
\begin{center}
    \includegraphics[width=0.49\linewidth]{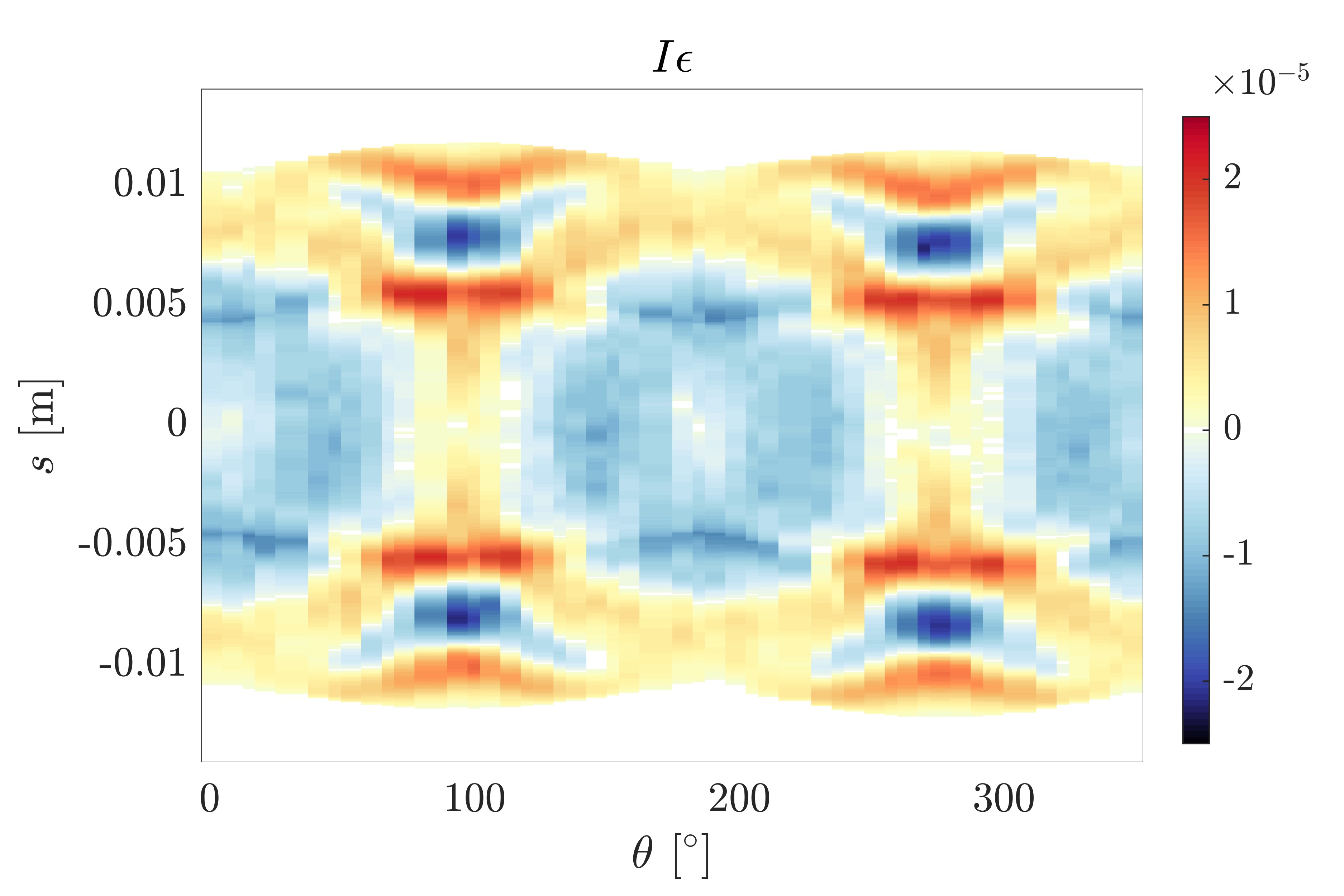}
    \begin{tikzpicture}
        \draw (-10,0) -- (-10,4.5);
    \end{tikzpicture}
    \includegraphics[width=0.49\linewidth]{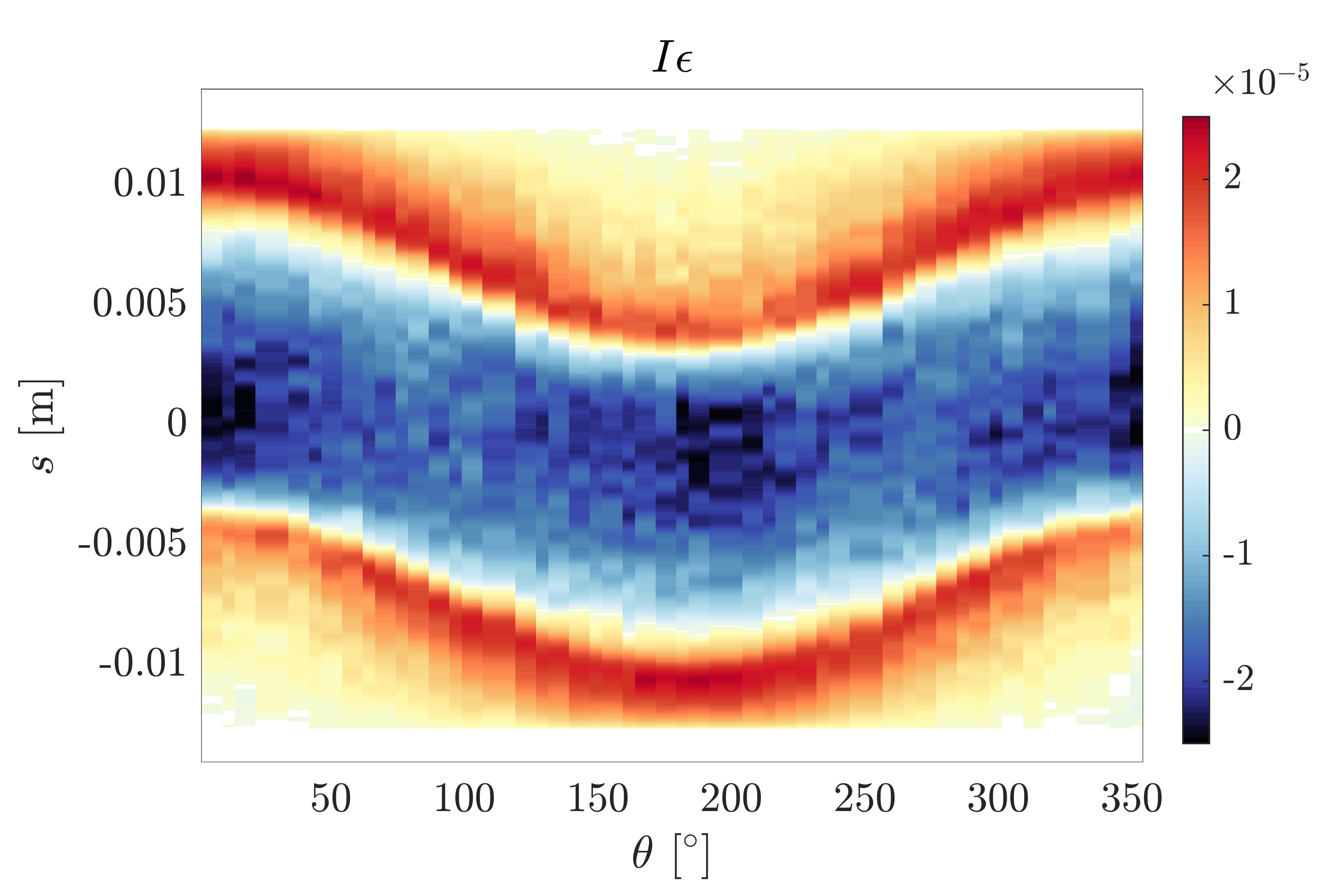} 
    \put(-380,115){(a)} \put(-185,115){(b)}\\
    \includegraphics[width=0.24\linewidth]{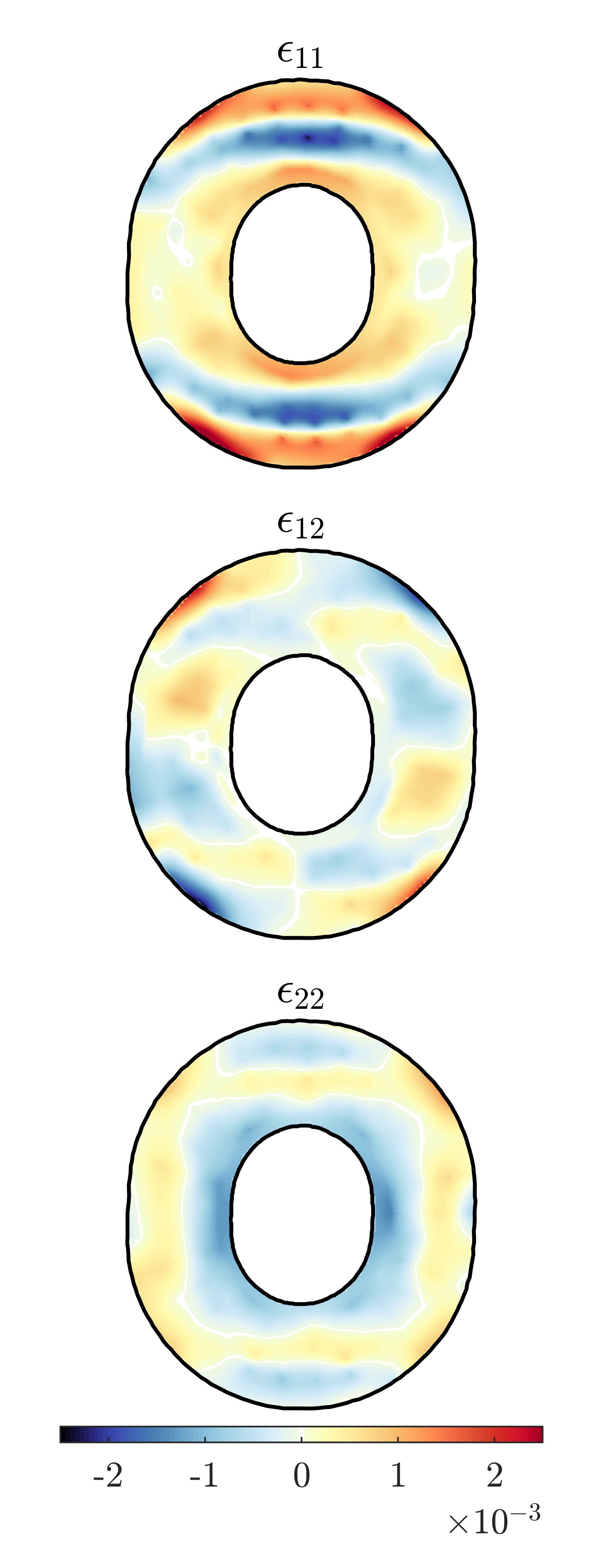}
    \includegraphics[width=0.24\linewidth]{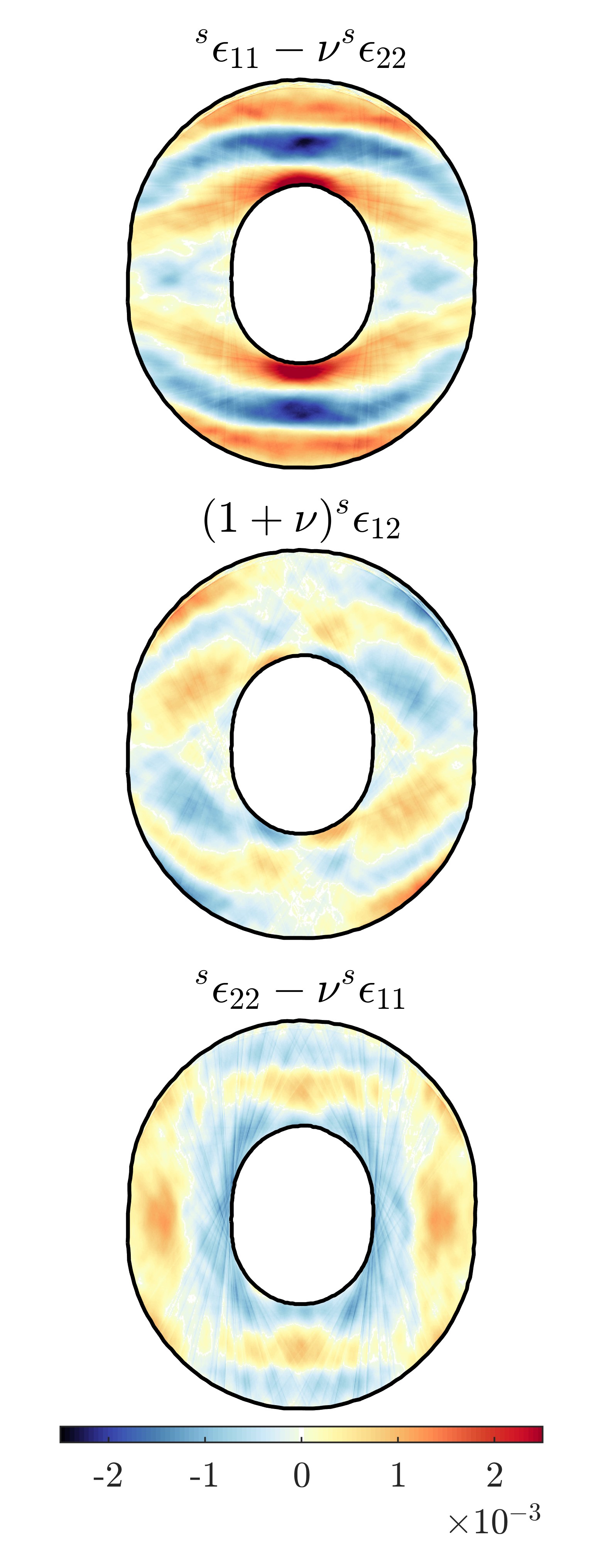}
    \begin{tikzpicture}
        \draw (-10,0) -- (-10,8.5);
    \end{tikzpicture}
    \includegraphics[width=0.24\linewidth]{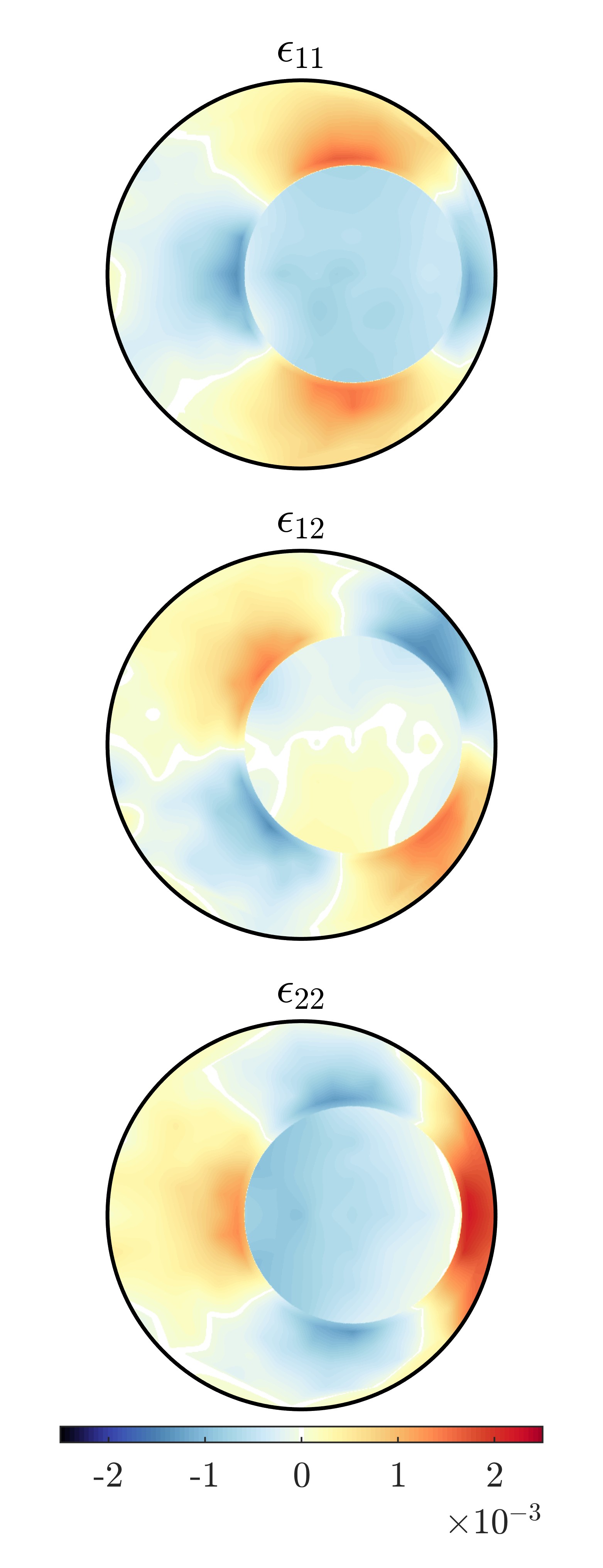}
    \includegraphics[width=0.24\linewidth]{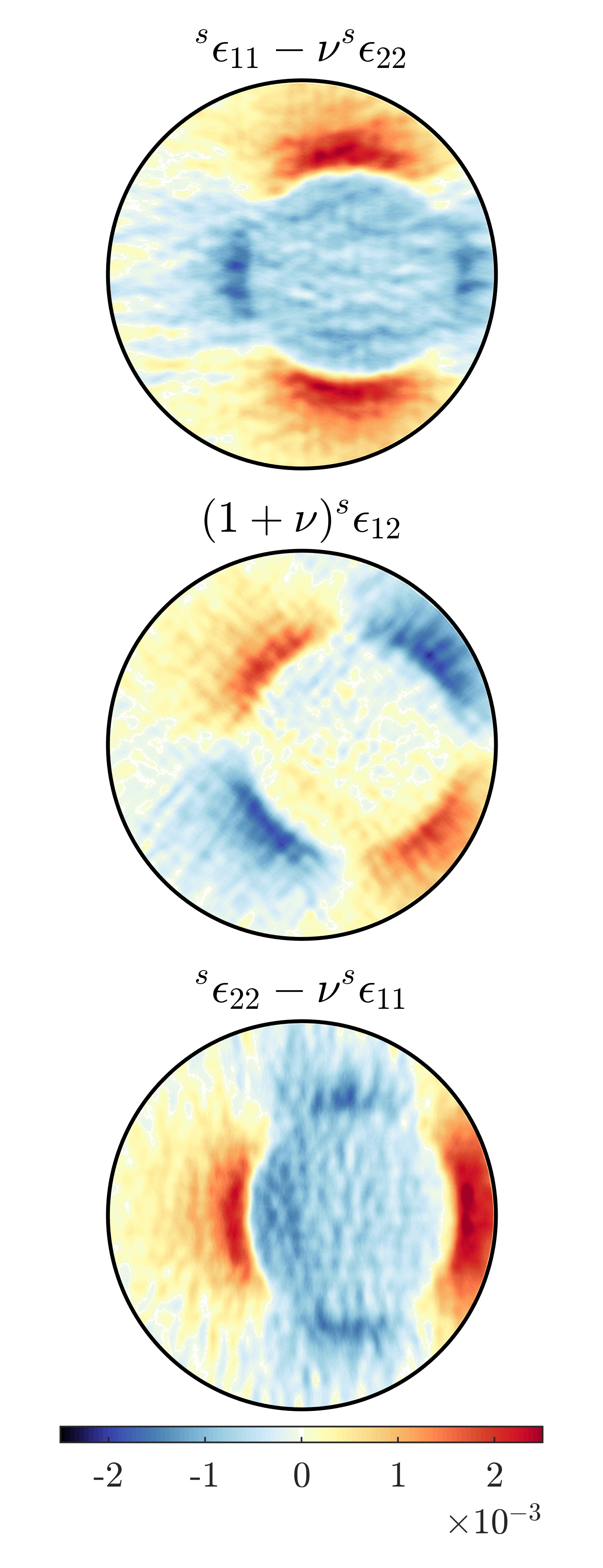}
    \put(-380,225){(c)} \put(-282.5,225){(d)} \put(-183.5,225){(e)} \put(-85,225){(f)}
    \caption{Reconstruction of residual strain fields from real data. (a) and (b) Measured LRT data from the crushed-ring and offset ring-and-plug samples using Bragg-edge strain imaging on the RADEN energy-resolved neutron imaging instrument \cite{gregg2018tomographic} (c) and (e) Reference measurements from each sample taken using traditional neutron diffraction based strain measurement techniques on the KOWARI engineering diffractometer (see \cite{hendriks2019robust}).  (d) and (f) reconstructed strain fields formed by the sum of the reconstructed solenoidal and recovered potential components. }
    \label{RealData}
\end{center}
\end{figure}
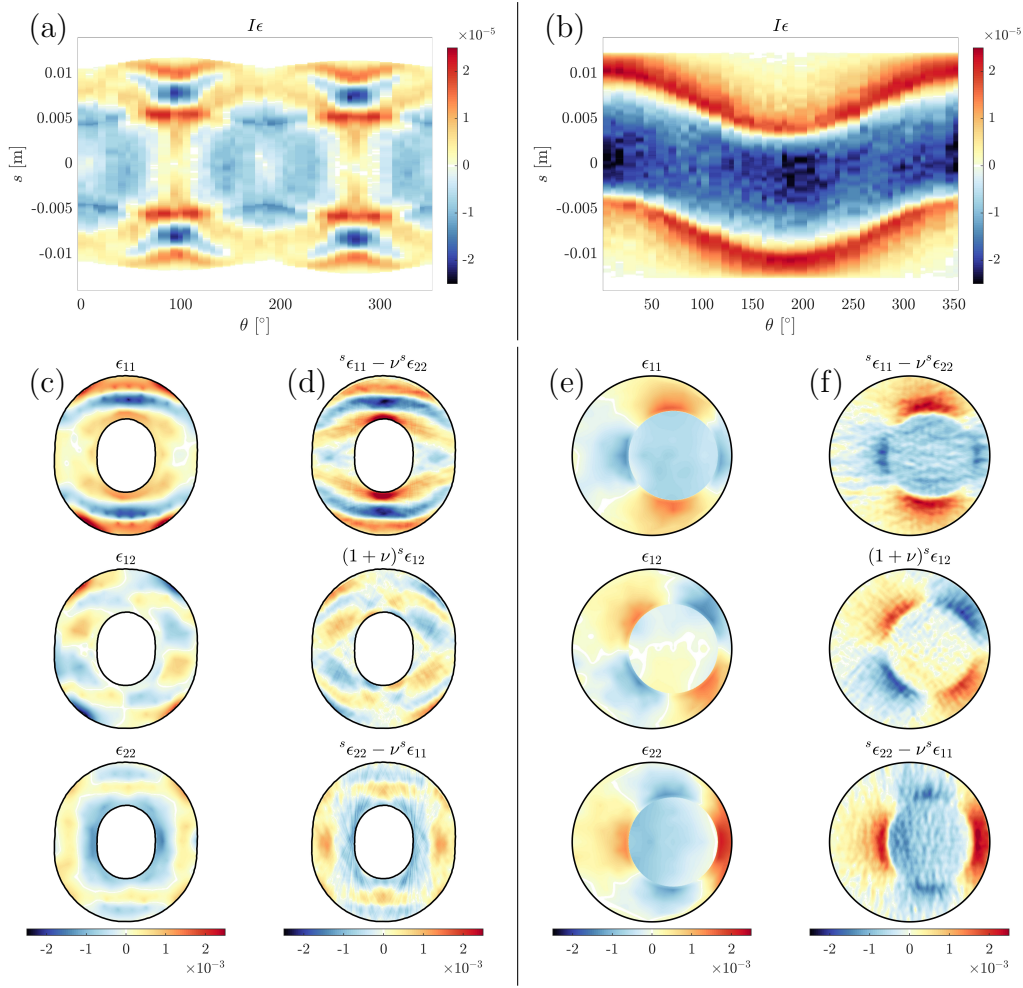 

Figure \ref{RealData}d and \ref{RealData}f show the results of the reconstruction based on Hooke's law compared to traditional neutron diffraction based strain measurements from the KOWARI engineering diffractometer at the Australian Centre for Neutron Scattering within the Australian Nuclear Science and Technology Organisation \cite{kirstein2010kowari}.  This reference data (Figure \ref{RealData}c and \ref{RealData}e) is in the form of interpolated/inferred fields computed from scattered measurements using a technique that guarantees equilibrium is satisfied at each point \cite{hendriks2019robust}.

Overall the reconstruction has performed well in terms of overall magnitude and distribution within the limits of resolution.  In particular, the reconstructions show remarkable similarity to that of previous work from the same data by Gregg \textit{et al} \cite{gregg2018tomographic} based on constrained least squares optimisation of Fourier basis functions.

\section{Conclusion}

A direct link has been established between the concept of Airy stress potentials in two-dimensional elastic systems and the standard Helmholtz decomposition at the heart of the LRT and its null space. For homogeneous, isotropic materials under plane-strain or plane-stress conditions, when the stress field satisfies equilibrium and has zero boundary traction, then the Helmholtz decomposition of the strain field can be written in terms of an Airy stress potential allowing for identification of the solenoidal and potential parts, which will have compact support.

Through this lens, direct approaches for the reconstruction of two-dimensional elastic strain fields from LRT data have been developed and demonstrated.  We show that a tensorial version of standard FBP recovers the solenoidal (divergence free) component of the strain field, which can then be used to determine the original field through the application of Hooke's law or a process involving the numerical solution of a standard elasticity problem.  In simulation, both approaches were found to be robust to measurement noise.  Both approaches also performed well on real experimental data.

From this perspective, it was also possible to identify the result of standard scalar FBP when applied to LRT measurement as the trace of the solenoidal component.  In some situations (e.g. plane-stress or plane-strain) this can be related to the trace of the stress tensor, however in general, more information is required to bring meaning to such a reconstruction in a three-dimensional system.

\section{Acknowledgements}

This work is supported by the Australian Research Council through a Discovery Project Grant (DP170102324).  Access to the RADEN and KOWARI instruments was made possible through the respective user access programs of J-PARC and ANSTO (J-PARC Long Term Proposal 2017L0101 and ANSTO Program Proposal PP6050).

Contributions from W Lionheart and S Holman were supported by the Engineering and Physical Sciences Research Council through grant EP/V007742/1.

Contributions from A Polyakova and I Svetov were supported by the framework of the government assignment of the Sobolev Institute of Mathematics, project FWNF-2022-0009.

Contributions from Matias Courdurier were partially supported by ANID Millennium Science Initiative Program through Millennium Nucleus for Applied Control and Inverse Problems NCN19-161.

The authors would also like to thank the Isaac Newton Institute for Mathematical Sciences  for support and hospitality during the program Rich and Non-linear Tomography: A Multidisciplinary Approach when work on this paper was undertaken. This program was supported by EPSRC grant number EP/R014604/1.  

While in Cambridge, all authors received support from the Simons Foundation. C Wensrich would also like to thank Clare Hall for their support and hospitality over this period.

\bibliographystyle{elsarticle-num} 
\bibliography{References}

\end{document}